\documentclass[12pt,a4paper,reqno]{amsart}
\usepackage{amsmath}
\usepackage{amsfonts}
\usepackage{amssymb}
\usepackage{amscd}
\usepackage{mathbbol}
\numberwithin{equation}{section}

\theoremstyle{plain}

\newtheorem{theorem}[subsection]{Theorem}
\newtheorem{proposition}[subsection]{Proposition}
\newtheorem{lemma}[subsection]{Lemma}
\newtheorem{corollary}[subsection]{Corollary}
\newtheorem{conjecture}[subsection]{Conjecture}

\theoremstyle{definition}

\theoremstyle{remark}

\newtheorem*{remark}{Remark}

\renewcommand{\leq}{\leqslant}
\renewcommand{\geq}{\geqslant}

\newsavebox{\proofbox}
\savebox{\proofbox}{\begin{picture}(7,7)%
  \put(0,0){\framebox(7,7){}}\end{picture}}

\def\Z{\mathbb{Z}}
\def\R{\mathbb{R}}

\def\C{\mathbb{C}}
\def\N{\mathbb{N}}

\def\cosec{{\operatorname{cosec}}}
\def\eps{\varepsilon}

\newcommand{\1}{\Eins}

\begin{document}

\title{Every odd number greater than $1$ is the sum of at most five primes}

\author{Terence Tao}
\address{UCLA Department of Mathematics, Los Angeles, CA 90095-1596.
}
\email{tao@math.ucla.edu}

\begin{abstract} 
We prove that every odd number $N$ greater than $1$ can be expressed as the sum of at most five primes, improving the result of Ramar\'e that every even natural number can be expressed as the sum of at most six primes.  We follow the circle method of Hardy-Littlewood and Vinogradov, together with Vaughan's identity; our additional techniques, which may be of interest for other Goldbach-type problems, include the use of smoothed exponential sums and optimisation of the Vaughan identity parameters to save or reduce some logarithmic losses, the use of multiple scales following some ideas of Bourgain, and the use of Montgomery's uncertainty principle and the large sieve to improve the $L^2$ estimates on major arcs.  Our argument relies on some previous numerical work, namely the verification of Richstein of the even Goldbach conjecture up to $4 \times 10^{14}$, and the verification of van de Lune and (independently) of Wedeniwski of the Riemann hypothesis up to height $3.29 \times 10^9$.
\end{abstract}

\maketitle

\section{Introduction}

Two of most well-known conjectures in additive number theory are the even and odd Goldbach conjectures, which we formulate as follows\footnote{The odd Goldbach conjecture is also often formulated in an almost equivalent (and slightly stronger) fashion as the assertion that every odd number greater than seven is the sum of three odd primes.}:

\begin{conjecture}[Even Goldbach conjecture]\label{even-gold} Every even natural number $x$ can be expressed as the sum of at most two primes.
\end{conjecture}

\begin{conjecture}[Odd Goldbach conjecture]\label{odd-gold} Every odd number $x$ larger than $1$ can be expressed as the sum of at most three primes.
\end{conjecture}

It was famously established by Vinogradov \cite{vinogradov}, using the Hardy-Littlewood circle method, that the odd Goldbach conjecture holds for all sufficiently large odd $x$.  Vinogradov's argument can be made effective, and various explicit thresholds for ``sufficiently large'' have been given in the literature; in particular, Chen and Wang \cite{cw2} established the odd Goldbach conjecture for all $x \geq \exp(\exp(11.503)) \approx \exp(99012)$, and Liu \& Wang \cite{liu} subsequently extended this result to the range $x \geq \exp(3100)$.  At the other extreme, by combining Richstein's numerical verification \cite{richstein} of the even Goldbach conjecture for $x \leq 4 \times 10^{14}$ with effective short intervals containing primes (based on a numerical verification of the Riemann hypothesis by van de Lune and Wedeniwski \cite{wedeniwski}), Ramar\'e and Saouter \cite{saouter} verified the odd Goldbach conjecture for $n \leq 1.13 \times 10^{22} \approx \exp(28)$.  By using subsequent numerical verifications of both the even Goldbach conjecture and the Riemann hypothesis, it is possible to increase this lower threshold somewhat, but there is still a very significant gap between the lower and upper thresholds for which the odd Goldbach conjecture is known\footnote{We remark however that the odd Goldbach conjecture is known to be true assuming the generalised Riemann hypothesis; see \cite{desh}.}.

To explain the reason for this, let us first quickly recall how Vinogradov-type theorems are proven.  To represent a number $x$ as the sum of three primes, it suffices to obtain a sufficiently non-trivial lower bound for the sum
$$ \sum_{n_1,n_2,n_3: n_1+n_2+n_3 = x} \Lambda(n_1) \Lambda(n_2) \Lambda(n_3)$$
where $\Lambda$ is the von Mangoldt function (see Section \ref{notation-sec} for definitions).  By Fourier analysis, we may rewrite this expression as the integral
\begin{equation}\label{roza}
 \int_{\R/\Z} S(x,\alpha)^3 e(-x\alpha)\ d\alpha
\end{equation}
where $e(x) := e^{2\pi ix}$ and $S(x,\alpha)$ is the exponential sum
$$ S(x,\alpha) := \sum_{n \leq x} \Lambda(n) e(n\alpha).$$
The objective is then to obtain sufficiently precise estimates on $S(x,\alpha)$ for large $x$.  Using the Dirichlet approximation theorem, one can approximate $\alpha = \frac{a}{q} + \beta$ for some natural number $1 \leq q \leq Q$, some integer $a$ with $(a,q)=1$, and some $\beta$ with $|\beta| \leq \frac{1}{qQ}$, where $Q$ is a threshold (somewhat close to $N$) to be chosen later.  Roughly speaking, one then divides into the \emph{major arc} case when $q$ is small, and the \emph{minor arc} case when $q$ is large (in practice, one may also subdivide these cases into further cases depending on the precise sizes of $q$ and $\beta$).  In the major arc case, the sum $S(x,\alpha)$ can be approximated by sums such as
$$ S(x,a/q) = \sum_{n \leq x} \Lambda(n) e(an/q),$$
which can be controlled by a suitable version of the prime number theorem in arithmetic progressions, such as the Siegel-Walfisz theorem.  In the minor arc case, one can instead follow the methods of Vinogradov, and use truncated divisor sum identities (such as (variants of) Vaughan's identity, see Lemma \ref{vaughan-lemma} below) to rewrite $S(x,\alpha)$ into various ``type I'' sums such as
$$ \sum_{d \leq U} \mu(d) \sum_{n \leq x/d} \log n e(\alpha dn)$$
and ``type II'' sums such as
\begin{equation}\label{type-ii}
 \sum_{d>U} \sum_{w>V} \mu(d) (\sum_{b|w:b>V} \Lambda(b)) e(\alpha dw)
\end{equation}
which one then estimates by standard linear and bilinear exponential sum tools, such as the Cauchy-Schwarz inequality.

A typical estimate in the minor arc case takes the form
\begin{equation}\label{sxa}
 |S(x,\alpha)| \ll \left(\frac{x}{\sqrt{q}} + \frac{x}{\sqrt{x/q}} + x^{4/5}\right) \log^3 x; 
\end{equation}
see e.g. \cite[Theorem 13.6]{iwaniec}.  An effective version of this estimate (with slightly different logarithmic powers) was given by Chen \& Wang \cite{cw}; see \eqref{Cwb} below.  Note that one cannot hope to obtain a bound for $S(x,\alpha)$ that is better than $\frac{x\sqrt{\phi(q)}}{q}$ without making progress on the Siegel zero problem (in order to decorrelate $\Lambda$ from the quadratic character with conductor $q$); see \cite{ramare} for further discussion.  In a similar spirit, one should not expect to obtain a bound better than $\frac{x}{\sqrt{x/q}}$ unless one exploits a non-trivial error term in the prime number theorem for arithmetic progressions (or equivalently, if one shows that $L$-functions do not have a zero on the line $\Re s = 1$), as one needs to decorrelate\footnote{Indeed, suppose for sake of heuristic argument that the Riemann zeta function had a zero very close to $1 + ix/q$, then the von Mangoldt explicit formula $\sum_{n \leq x} \Lambda(n)$ would contain a term close to $-\sum_{n \leq x} n^{-ix/q}$, which suggests upon summation by parts that $S(x,\alpha)$ would contain a term close to $-\sum_{n \leq x} n^{-ix/q} e(\alpha x)$, which can be of size comparable to $x/\sqrt{x/q}$ when $\alpha$ is close to (say) $1/q$.  A similar argument involving Dirichlet $L$-functions also applies when $\alpha$ is close to other multiples of $1/q$.  This suggests that one would need to use zero-free regions for zeta functions and $L$-functions in orderto improve upon the $\frac{x}{\sqrt{x/q}}$ bound.} $\Lambda$ from Archimedean characters $n^{it}$ with $t$ comparable to $x/q$, (or from combined characters $\chi(n) n^{it}$).

If one combines \eqref{sxa} with the $L^2$ bound
\begin{equation}\label{sxa2}
 \int_{\R/\Z} |S(x,\alpha)|^2\ d\alpha \ll x \log x
\end{equation}
arising from the Plancherel identity and the prime number theorem, one can obtain satisfactory estimates for all minor arcs with $q \gg \log^8 x$ (assuming that $Q$ was chosen to be significantly less than $x/\log^8 x$).  To finish the proof of Vinogradov's theorem, one thus needs to obtain good prime number theorems in arithmetic progressions whose modulus $q$ can be as large as $\log^8 x$.  While explicitly effective versions of such theorems exist (see e.g. \cite{mccurley}, \cite{liu-0}, \cite{rumely}, \cite{dusart}, \cite{kadiri}), their error term decays very slowly (and is sensitive to the possibility of a Siegel zero for one exceptional modulus $q$); in particular, errors of the form $O( x \exp(-c \sqrt{\log x}) )$ for some explicit but moderately small constant $c>0$ are typical.  Such errors can eventually overcome logarithmic losses such as $\log x$, but only for extremely large $x$, e.g. for $x$ much larger than $10^{100}$, which can be viewed as the principal reason why the thresholds for Vinogradov-type theorems are so large.

It is thus of interest to reduce the logarithmic losses in \eqref{sxa}, particularly for moderately sized $x$ such as $x \sim 10^{30}$, as this would reduce the range\footnote{It seems however that one cannot eliminate the major arcs entirely.  In particular, one needs to prevent the majority of the primes from concentrating in the quadratic non-residues modulo $q$ for $q=3$, $q=4$, or $q=5$, as this would cause the residue class $0 \mod q$ to have almost no representations as the sum of three primes.} of moduli $q$ that would have to be checked in the major arc case, allowing for stronger prime number theorems in arithmetic progressions to come into play.  In particular, the numerical work of Platt \cite{platt} has established zero-free regions for Dirichlet $L$-functions of conductor up to $10^5$ (building upon earlier work of Rumely \cite{rumely} who obtained similar results up to conductor $10^2$), and it is likely that such results would be useful in future work on Goldbach-type problems.

Several improvements or variants to \eqref{sxa} are already known in the literature.  For instance, prior to the introduction of Vaughan's identity in \cite{vaughan}, Vinogradov \cite{vin2} had already established a bound of the shape
$$  |S(x,\alpha)| \ll \left(\frac{x}{\sqrt{q}} + \frac{x}{\sqrt{x/q}} + \exp(-\tfrac{1}{2}\sqrt{\log x})\right) \log^{11/2} x$$
by sieving out small primes.  This was improved by Chen \cite{chen} and Daboussi \cite{dab} to obtain
$$  |S(x,\alpha)| \ll \left(\frac{x}{\sqrt{q}} + \frac{x}{\sqrt{x/q}}  + \exp(-\tfrac{1}{2}\sqrt{\log x}) \right) \log^{3/4} x (\log\log x)^{1/2}$$
and the constants made explicit (with some slight degradation in the logarithmic exponents) by Daboussi and Rivat \cite{dr}.  Unfortunately, due to the slow decay of the $\exp(-\tfrac{1}{2}\sqrt{\log x})$ term, this estimate only becomes non-trivial for $x \geq 10^{184}$ (see \cite[\S 8]{dr}), and is thus not of direct use for smaller values of $x$.

In \cite{buttkewitz}, Buttkewitz obtained the bound
\begin{equation}\label{xaq}
  |S(x,\alpha)| \ll_A \frac{x}{\sqrt{q}} \log^{1/4} x + \frac{x}{x/q} \log\log\log x + \frac{x}{\log^A x}
\end{equation}
for any $A \geq 1$ (assuming $\beta$ extremely small), using Lavrik's estimate \cite{lavrik} on the average error term in twin prime type problems, but the constants here are ineffective (as Lavrik's estimate uses the Siegel-Walfisz theorem).  Finally, in the ``weakly minor arc'' regime $\log q \leq \frac{1}{50} \log^{1/3} x$, Ramar\'e \cite{bombieri} used the Bombieri sieve to obtain an effective bound of the form
\begin{equation}\label{xaq-2}
|S(x,\alpha)| \ll \frac{x\sqrt{q}}{\phi(q)}
\end{equation}
which, as mentioned previously, is an essentially sharp effective bound in the absence of any progress on the Siegel zero problem. Unfortunately, the constants in \cite{bombieri} are not computed explicitly, and seem to be far too large to be useful for values of $x$ such as $10^{30}$.

In this paper we will not work directly with the sums $S(x,\alpha)$, but instead with the smoothed out sums
$$
S_{\eta,q_0}(x,\alpha) := \sum_n \Lambda(n) e(\alpha n) \1_{(n,q_0)=1} \eta(n/x)
$$
for some (piecewise) smooth functions $\eta: \R \to \C$ and a modulus $q_0$.  The modulus $q_0$ is only of minor technical importance, and should be ignored at a first reading (see Lemma \ref{eta-smash} for a precise version of this statement).  For sake of explicit constants we will often work with a specific choice of cutoff $\eta$, namely the cutoff 
\begin{equation}\label{Eta0-def}
 \eta_0(t) := 4 (\log 2 - |\log 2t|)_+,
\end{equation}
which is a Lipschitz continuous cutoff of unit mass supported on the interval $[1/4,1]$.  The use of smooth cutoffs to improve the behaviour of exponential sums is of course well established in the literature (indeed, the original work of Hardy and Littlewood \cite{hardy} on Goldbach-type problems already used smoothed sums).  The reason for this specific cutoff is that one has the identity
\begin{equation}\label{eta0}
\eta_0(dw/x) = 4 \int_0^\infty \1_{[\frac{x}{2W}, \frac{x}{W}]}(d) \1_{[\frac{W}{2}, W]}(w)\ \frac{dW}{W}
\end{equation}
for any $d,w,x > 0$, which will be convenient for factorising the Type II sums into a tractable form.  In some cases we will take $q_0$ to equal $2$, in order to restrict all sums to be over odd numbers, rather than all natural numbers; this has the effect of saving a factor of two in the explicit constants.  Because of this, though, it becomes natural to approximate $4\alpha$ by a rational number $a/q$, rather than $\alpha$ itself.

Our main exponential sum result can be stated as follows.

\begin{theorem}[Exponential sum estimate]\label{expsum}
Let $x \geq 10^{20}$ be a real number, and let $4\alpha = \frac{a}{q} + \beta$ for some natural number $100 \leq q \leq x/100$ with $(a,q)=1$ and some $\beta \in [-1/q^2,1/q^2]$.  Let $q_0$ be a natural number, such that all prime factors of $q_0$ do not exceed $\sqrt{x}$. Then
\begin{equation}\label{Sax}
 |S_{\eta_0,q_0}(x,\alpha)| \leq \left(0.14 \frac{x}{\sqrt{q}} + 0.64 \frac{x}{\sqrt{x/q}} + 0.15 x^{4/5}\right) \log x (\log x + 11.3)
\end{equation}
If $q \leq x^{1/3}$, one has the refinement
\begin{equation}\label{Sax-2}
|S_{\eta_0,q_0}(x,\alpha)| \leq 0.5 \frac{x}{q} (\log 2x) (\log 2x + 15) + 0.31 \frac{x}{\sqrt{q}} \log q (\log q + 8.9)
\end{equation}
and similarly when $q \geq x^{2/3}$, one has the refinement
\begin{equation}\label{Sax-3}
|S_{\eta_0,q_0}(x,\alpha)| \leq 3.12 \frac{x}{x/q} (\log 2x) (\log q + 8) + 1.19 \frac{x}{\sqrt{x/q}} \log \frac{x}{q} (\log \frac{x}{q} + 2.3).
\end{equation}
If $q \geq x^{2/3}$ and $a= \pm 1$, one has the further refinement
\begin{equation}\label{lab}
 |S_{\eta_0,q_0}(x,\alpha)| \leq 9.73 \frac{x}{(x/q)^2} \log^2 x + 1.2 \frac{x}{\sqrt{x/q}} \log \frac{x}{q} (\log \frac{x}{q} + 2.4).
\end{equation}
\end{theorem}

The first estimate \eqref{Sax} is basically a smoothed out version of the standard estimate \eqref{sxa}, with the Lipschitz nature of the cutoff $\eta_0$ being responsible for the saving of one logarithmic factor, and will be proven by basically the same method (i.e. Vaughan's identity, followed by linear and bilinear sum estimates of Vinogradov type); it can be compared for instance with the explicit estimate
\begin{equation}\label{Cwb}
|S_{\1_{[0,1]},1}(x,4\alpha)|
\leq 0.177 \frac{x}{\sqrt{q}} \log^3 x + 0.08 \frac{x}{\sqrt{x/q}} \log^{3.5} x + 3.8 x^{4/5} \log^{2.2} x
\end{equation}
of Chen \& Wang \cite{cw}, rewritten in our notation.  In practice, for $x$ of size $10^{30}$ or so, the estimate \eqref{sxa} improves upon the Chen-Wang estimate by about one to two orders of magnitude, though this is admittedly for a smoothed version of the sum considered in \cite{cw}.

The estimate \eqref{Sax} is non-trivial in the regime $\log^4 x \ll q \ll x / \log^4 x$.  The key point, though, is that one can obtain further improvement over \eqref{Sax} in the most important regimes when $q$ is close to $1$ or to $x$, by reducing the argument of the logarithm (or by squaring the denominator); indeed, \eqref{Sax} when combined with \eqref{Sax-2}, \eqref{Sax-3} is now non-trivial in the larger range $\log^2 x \ll q \ll x / \log^2 x$, and with \eqref{lab} one can extend the upper threshold of this range from $x/\log^2 x$ to $x/\log x$, at least when $a=\pm 1$.

The bounds in Theorem \ref{expsum} are basically achieved by optimising in the cutoff parameters $U,V$ in Vaughan's identity, and (in the case of \eqref{lab}) a second integration by parts, exploiting the fact that the derivative of $\eta_0$ has bounded total variation.  Asymptotically, the bounds here are inferior to those in \eqref{xaq}, \eqref{xaq-2}, but unlike those estimates, the constants here are reasonable enough that the bounds are useful for medium-sized values of $x$, for instance for $x$ between $10^{30}$ and $10^{1300}$.  There is scope for further improvement\footnote{Indeed, since the release of an initial preprint of this paper, such improvements to the above bounds have been achieved in \cite{helf}.  The author expects however that in the case of most critical interest, namely when $q$ is very close to $1$ or to $x$, the Vaughan identity-based methods in the above theorem are not the most efficient way to proceed.  Thus, one can imagine in future applications that Theorem \ref{expsum} (or variants thereof) could be used to eliminate from consideration all values of $q$ except those close to $1$ and $x$, and then other methods (e.g. those based on the zeroes of Dirichlet L-functions, or more efficient identities than the Vaughan identity) could be used to handle those cases.} in the exponential sum bounds in these ranges by eliminating or reducing more of the logarithmic losses (and we did not fully attempt to optimise all the explicit constants), but the bounds indicated above will be sufficient for our applications.  

We will actually prove a slightly sharper (but more technical) version of Theorem \ref{expsum}, with two additional parameters $U$ and $V$ that one is free to optimise over; see Theorem \ref{strong-minor} below.

As an application of these exponential sum bounds, we establish the following result:

\begin{theorem}\label{main}  Every odd number $x$ larger than $1$ can be expressed as the sum of at most five primes.
\end{theorem}

This improves slightly upon a previous result of Ramar\'e \cite{ramare}, who showed (by a quite different method) that every even natural number can be expressed as the sum of at most six primes.  In particular, as a corollary of this result we may lower the upper bound on Shnirelman's constant (the least number $k$ such that all natural numbers greater than $1$ can be expressed as the sum of at most $k$ primes) from seven to six (note that the even Goldbach conjecture would imply that this constant is in fact three, and is in fact almost equivalent to this claim).  We remark that Theorem \ref{main} was also established under the assumption of the Riemann hypothesis by Kaniecki \cite{kaniecki} (indeed, by combining his argument with the numerical work in \cite{richstein}, one can obtain the stronger claim that any even natural number is the sum of at most four primes).

Our proof of Theorem \ref{main} also establishes that every integer $x$ larger than $8.7 \times 10^{36}$ can be expressed as the sum of three primes and a natural number between $2$ and $4 \times 10^{14}$; see Theorem \ref{mod} below.

To prove Theorem \ref{main}, we will also need to rely on two numerically verified results in addition to Theorem \ref{expsum}:

\begin{theorem}[Numerical verification of Riemann hypothesis]\label{rh-check}  Let $T_0 := 3.29 \times 10^9$.  Then all the zeroes of the Riemann zeta function $\zeta$ in the strip $\{ s: 0 < \Re(s) < 1; 0 \leq \Im(s) \leq T_0 \}$ lie on the line $\Re(s)=1/2$.  Furthermore, there are at most $10^{10}$ zeroes in this strip.
\end{theorem}

\begin{proof}  This was achieved independently by van de Lune (unpublished), by Wedeniwski \cite{wedeniwski}, by Gourdon \cite{gourdon}, and by Platt \cite{platt-12}.  Indeed, the results of Wedeniwski allow one to take $T_0$ as large as $5.72 \times 10^{10}$, and the results of Gourdon allow one to take $T_0$ as large as $2.44 \times 10^{12}$; using interval arithmetic, Platt also obtained this result with $T_0$ as large as $3.06 \times 10^{10}$.  (Of course, in these latter results there will be more than $10^{10}$ zeroes.)  However, we will use the more conservative value of $T_0 = 3.29 \times 10^9$ in this paper as it suffices for our purposes, and has been verified by four independent numerical computations.  
\end{proof}

\begin{theorem}[Numerical verification of even Goldbach conjecture]\label{gh-check}  Let $N_0 := 4 \times 10^{14}$.  Then every even number between $4$ and $N_0$ is the sum of two primes.
\end{theorem}

\begin{proof}  This is the main result of Richstein \cite{richstein}.  A subsequent (unpublished) verification of this conjecture by the distributed computing project of Oliveira e Silva \cite{silva} allows one to take $N_0$ as large as $2.6 \times 10^{18}$ (with the value $N_0=10^{17}$ being double-checked), but again we shall use the more conservative value of $N_0 = 4 \times 10^{14}$ in this paper as it suffices for our purposes, and has been verified by three independent numerical computations.
\end{proof}

The proof of Theorem \ref{main} is given in Section \ref{final-sec} proceeds according to the circle method with smoothed sums as discussed earlier, but with some additional technical refinements which we now discuss.   The first refinement is to take advantage of Theorem \ref{gh-check} to reduce the five-prime problem to the problem of representing a number $x$ as the sum of three primes $n_1,n_2,n_3$ and a number between $2$ and $N_0$.  As far as the circle method is concerned, this effectively restricts the frequency variable $\alpha$ to the arc $\{ \alpha: \|\alpha\|_{\R/\Z} \ll 1/N_0 \}$.   At the other end of the spectrum, by using Theorem \ref{rh-check} and the von Mangoldt explicit formula one can control quite precisely the contribution of the major arc $\{ \alpha: \|\alpha\|_{\R/\Z} \ll T_0/x \}$; see Proposition \ref{rh}.  Thus leaves only the ``minor arc'' 
\begin{equation}\label{minor}
 \frac{T_0}{x} \ll \|\alpha\|_{\R/\Z} \ll \frac{1}{N_0}
\end{equation}
that remains to be controlled.  

By using Theorem \ref{gh-check} and Theorem \ref{rh-check} (or more precisely, an effective prime number theorem in short intervals derived from Theorem \ref{rh-check} due to Ramar\'e and Saouter \cite{saouter}), we will be able to assume that $x$ is moderately large (and specifically, that $x \geq 8.7 \times 10^{36}$).  By using existing Vinogradov-type theorems, we may also place a large upper bound on $x$; we will use the strongest bound in the literature, namely the bound $x \leq \exp(3100)$ of\footnote{The numerology in our argument is such that one could also use the weaker bound $x \leq \exp(\exp(11.503))$ provided by Chen \& Wang \cite{cw2}, provided one assumed the even Goldbach conjecture to be verified up to $N_0 = 10^{17}$, a fact which has been double-checked in \cite{silva}.  Alternatively, if one uses the very recent minor arc bounds in \cite{helf} that appeared after the publication of this paper, then no Vinogradov type theorem is needed at all to prove our result, as the bounds in \cite{helf} do not contain any factors of $\log x$ that need to be bounded.} Liu \& Wang \cite{liu}.  In particular, $\log x$ is relatively small compared to the quantities $N_0$ and $T_0$, allowing one to absorb a limited number of powers of $\log x$ in the estimates.

It remains to obtain good $L^2$ and $L^\infty$ estimates on the minor arc region \eqref{minor}.  We will of course use Theorem \ref{expsum} for the $L^\infty$ estimates.  A direct application of the Plancherel identity would cost a factor of $\log x$ in the $L^2$ estimates, which turns out to be unacceptable.  One can use the uncertainty principle of Montgomery \cite{montgomery} to cut this loss down to approximately $2 \frac{\log x}{\log N_0}$ (see Corollary \ref{uplow}), but it turns out to be more efficient still to use a large sieve estimate of Siebert \cite{sie} on the number of prime pairs $(p,p+h)$ less than $x$ for various $h$ to obtain an $L^2$ estimate which only loses a factor of $8$ (see Proposition \ref{meso}). 

In order to nearly eliminate some additional ``Archimedean'' losses arising from convolving together various cutoff functions $\eta$ on the real line $\R$, we will use a trick of Bourgain \cite{bourgain}, and restrict one of the three summands $n_1,n_2,n_3$ to a significantly smaller order of magnitude (of magnitude $x/K$ instead of $x$ for some $K$, which we will set to be $10^3$, in order to be safely bounded away from both $1$ and $T_0$).  By estimating the exponential sums associated to $n_1,n_2$ in $L^2$ and the sum associated to $n_3$ in $L^\infty$, one can avoid almost all Archimedean losses.

As it turns out, the combination of all of these tricks, when combined with the exponential sum estimate and the numerical values of $N_0$ and $T_0$, are sufficient to establish Theorem \ref{main}.  However, there is one final trick which could be used to reduce certain error terms further, namely to let $K$ vary over a range (e.g. from $10^3$ to $10^6$) instead of being fixed, and average over this parameter.  This turns out to lead to an additional saving (of approximately an order of magnitude) in the weakly minor arc case when $\alpha$ is slightly larger than $T_0/x$.  While we do not actually utilise this trick here as it is not needed, it may be useful in other contexts, and in particular in improving the known upper threshold for Vinogradov's theorem.

\subsection{Acknowledgments}

The author is greatly indebted to Ben Green and Harald Helfgott for several discussions on Goldbach-type problems, and to Julia Wolf, Westin King, Mark Kerstein, and several anonymous readers of my blog for corrections.  The author is particularly indebted to the anonymous referee for detailed comments, corrections, and suggestions.  The author is supported by NSF grant DMS-0649473.

\section{Notation}\label{notation-sec}

All summations in this paper will be over the natural numbers ${\bf N} = \{ 1,2,3,\ldots\}$ unless otherwise indicated, with the exceptions of sums over the index $p$, which are always understood to be over primes.

We use $(a,b)$ to denote the greatest common divisor of two natural numbers, and $[a,b]$ for the least common multiple.  We write $a|b$ to denote the assertion that $a$ divides $b$.

Given a statement $E$, we use $\1_E$ to denote its indicator, thus $\1_E$ equals $1$ when $E$ is true and $0$ otherwise.  Thus, for instance $\1_{(n,q)=1}$ is equal to $1$ when $n$ is coprime to $q$, and equal to $0$ otherwise.

We will use the following standard arithmetic functions.  If $n$ is a natural number, then
\begin{itemize}
\item $\Lambda(n)$ is the von Mangoldt function of $n$, defined to equal $\log p$ when $n$ is a power of a prime $p$, and zero otherwise.
\item $\mu(n)$ is the M\"obius function of $n$, defined to equal $(-1)^k$ when $n$ is the product of $k$ distinct primes, and zero otherwise.
\item $\phi(n)$ is the Euler totient function of $n$, defined to equal the number of residue classes of $n$ that are coprime to $n$.
\item $\omega(n)$ is the number of distinct prime factors of $n$.
\end{itemize}

Given two arithmetic functions $f, g: \N \to \C$, we define the Dirichlet convolution $f \ast g: \N \to \C$ by the formula
$$ f \ast g(n) := \sum_{d|n} f(d) g(\frac{n}{d}),$$
thus for instance $\mu \ast 1(n) = \1_{n=0}$, $\Lambda \ast 1(n) = \log(n)$, and $\mu \ast \log(n) = \Lambda(n)$.

Given a positive real $Q$, we also define the \emph{primorial} $Q\sharp$ of $Q$ to be the product of all the primes up to $Q$:
$$ Q\sharp := \prod_{p \leq Q} p.$$
Thus, for instance $\1_{(n,Q\sharp)=1}$ equals $1$ precisely when $n$ has no prime factors less than or equal to $Q$.

We use the usual $L^p$ norms
$$ \|f\|_{L^p(\R)} := \left(\int_\R |f(x)|^p\ dx\right)^{1/p}$$
for $1 \leq p < \infty$, with $\|f\|_{L^\infty(\R)}$ denoting the essential supremum of $f$.  Similarly for other domains than $\R$ which have an obvious Lebesgue measure; we define the $\ell^p$ norms for discrete domains (such as the integers $\Z$) in the usual manner.

We use $\R/\Z$ to denote the unit circle.  By abuse of notation, any $1$-periodic function on $\R$ is also interpreted as a function on $\R/\Z$, thus for instance we can define $|\sin(\pi\alpha)|$ for any $\alpha \in \R/\Z$.  We let $\|\alpha\|_{\R/\Z}$ be the distance to the nearest integer for $\alpha \in \R$, and this also descends to $\R/\Z$. We record the elementary inequalities
\begin{equation}\label{sinx}
2 \| \alpha \|_{\R/\Z} \leq |\sin(\pi \alpha)| \leq \pi \| \alpha \|_{\R/\Z} \leq |\tan(\pi \alpha)|,
\end{equation}
valid for any $\alpha \in \R/\Z$.  In a similar vein, if $a \in\Z/q\Z$, we consider $\frac{a}{q}$ as an element of $\R/\Z$.

When considering a quotient $\frac{X}{Y}$ with a non-negative denominator $Y$, we adopt the convention that $\frac{X}{Y}=+\infty$ when $Y$ is zero.  Thus for instance $\frac{1}{\|\alpha\|_{\R/\Z}}$ is equal to $+\infty$ when $\alpha$ is an integer.

We will occasionally use the usual asymptotic notation $X = O(Y)$ or $X \ll Y$ to denote the estimate $|X| \leq CY$ for some unspecified constant $C$.  However, as we will need explicit bounds, we will more frequently (following Ramar\'e \cite{ramare}) use the exact asymptotic notation $X = {\mathcal O}^*(Y)$ to denote the estimate $|X| \leq Y$.  Thus, for instance, $X = Y + {\mathcal O}^*(Z)$ is synonymous with $|X-Y| \leq Z$.

If $F: \R \to \C$ is a smooth function, we use $F', F''$ to denote the first and second derivatives of $F$, and $F^{(k)}$ to denote the $k$-fold derivative for any $k \geq 0$.

If $x$ is a real number, we denote $e(x) := e^{2\pi i x}$, we denote $x_+ := \max(x,0)$, and we denote $\lfloor x \rfloor$ to be the greatest integer less than or equal to $x$.  We interpret the $x \mapsto x_+$ operation to have precedence over exponentiation, thus for instance $x_+^2$ denotes the quantity $\max(x,0)^2$ rather than $\max(x^2,0)$.

If $E$ is a finite set, we use $|E|$ to denote its cardinality.  If $I$ is an interval, we use $|I|$ to denote its length. We will also occasionally use translations $I+x:=\{ y+x: y\in I\}$ and dilations $\lambda I:= \{ \lambda y: y \in I\}$ of such an interval.

The natural logarithm function $\log$ takes precedence over addition and subtraction, but not multiplication or division; thus for instance
$$\log 2x + 15 = (\log(2x))+15.$$

\section{Exponential sum estimates}

We now record some estimates on linear exponential sums such as
$$ \sum_{n \in \Z} F(n) e(\alpha n)$$
for various smooth functions $F: \R \to \C$, as well as bilinear exponential sums such as
$$ \sum_{n \in \Z} \sum_{m \in \Z} a_n b_m e(\alpha nm)$$
for various sequences $(a_n)_{n \in \Z}$ and $(b_m)_{m \in \Z}$.  These bounds are standard (at least if one is only interested in bounds up to multiplicative constants), but we will need to make all constants explicit.  Also, in order to save a factor of two or so in the explicit bounds, we will frequently restrict the summation to odd integers by inserting weights such as $\1_{(n,2)=1}$, and will therefore need to develop variants of the standard bounds for this purpose.

We begin with some standard bounds on linear exponential sums with smooth cutoffs.

\begin{lemma}\label{lo}  Let $\alpha \in \R/\Z$, and let $F: \R \to \C$ be a smooth, compactly supported function.  Then we have the bounds
\begin{equation}\label{f0}
\sum_{n \in \Z} F(n) = \int_\R F(y)\ dy + {\mathcal O}^*\left(\tfrac{1}{2} \|F'\|_{L^1(\R)}\right)
\end{equation}
and
\begin{equation}\label{f1}
\left|\sum_{n \in \Z} F(n) e(\alpha n)\right| \leq \sum_n |F(n)| \leq \|F\|_{L^1(\R)} + \tfrac{1}{2} \|F'\|_{L^1(\R)}
\end{equation}
and
\begin{equation}\label{fk}
\left|\sum_{n \in \Z} F(n) e(\alpha n)\right| \leq \frac{1}{|2\sin(\pi \alpha)|^k} \| F^{(k)}\|_{L^1(\R)}
\end{equation}
for all natural numbers $k \geq 1$.
\end{lemma}

\begin{proof}  These bounds appear in \cite{gallagher} and in \cite[Lemma 1.1]{mont-topic}, but we reproduce the proof here for the reader's convenience.
From the fundamental theorem of calculus one has
$$F(n) = F(y) + {\mathcal O}^*\left(\int_{n}^{n+1/2} |F'(t)|\ dt\right)$$
for all $y \in [n,n+1/2]$ and 
$$F(n) = F(y) + {\mathcal O}^*\left(\int_{n-1/2}^{n} |F'(t)|\ dt\right)$$
for all $y \in [n-1/2,n]$; averaging over $y$, we conclude that
$$ F(n) = \int_{n-1/2}^{n+1/2} F(y)\ dy + \tfrac{1}{2} {\mathcal O}^*\left( \int_{n-1/2}^{n+1/2} |F'(t)|\ dt \right).$$
Summing over $n$, we obtain \eqref{f0}.  Taking $\ell^1$ norms instead, one obtains
$$ \sum_{n \in \Z} |F(n)| \leq \|F\|_{L^1(\R)} + \tfrac{1}{2} \|F'\|_{L^1(\R)}$$
giving \eqref{f1}.

Now we obtain the $k=1$ case of \eqref{fk}.
We may assume $\alpha \neq 0$.  By summation by parts, one has
$$
\sum_{n \in \Z} (F(n+1)-F(n)) e(\alpha n) = (1 - e(-\alpha)) \sum_{n \in \Z} F(n) e(\alpha n)$$
and thus
\begin{equation}\label{fang}
 \left|\sum_{n \in \Z} F(n) e(\alpha n)\right| = \frac{1}{2 |\sin(\pi \alpha)|} |\sum_{n \in \Z} (F(n+1)-F(n)) e(\alpha n)|.
\end{equation}
Since $|F(n+1)-F(n)| \leq \int_0^1 |F'(n+t)|\ dt$, we obtain the $k=1$ case of \eqref{fk}.  To obtain the higher cases, observe from \eqref{fang}, the fundamental theorem of calculus $F(n+1)-F(n) = \int_0^1 F'(n+t)\ dt$ and Minkowski's inequality that
$$ |\sum_{n \in \Z} F(n) e(\alpha n)| \leq \frac{1}{2 |\sin(\pi \alpha)|} |\sum_{n \in \Z} F'(n+t) e(\alpha n)|$$
for some $0 \leq t < 1$.  One can now deduce \eqref{fk} for general $k$ from the $k=1$ case by induction.
\end{proof}

We can save a factor of two by restricting to odd numbers:

\begin{corollary}\label{googly}  With the same hypotheses as Lemma \ref{lo}, we have
$$ |\sum_{n \in \Z} F(n) e(\alpha n) \1_{(2,n)=1}| \leq \tfrac{1}{2} \|F\|_{L^1(\R)} + \tfrac{1}{2} \|F'\|_{L^1(\R)}$$
and
$$ |\sum_{n \in \Z} F(n) e(\alpha n) \1_{(2,n)=1}| \leq \frac{1}{2|\sin(2\pi \alpha)|^k} \| F^{(k)}\|_{L^1(\R)} $$
for all $k \geq 1$.
\end{corollary}

\begin{proof}  We can rewrite
$$ \sum_{n \in \Z} F(n) e(\alpha n) \1_{(2,n)=1} = e(\alpha) \sum_{n \in \Z} F(2n+1) e(2\alpha n).$$
Applying the previous lemma with $\alpha$ replaced by $2\alpha$, and $F(x)$ replaced by $F(2x+1)$, we obtain the claim.
\end{proof}

We also record a continuous variant:

\begin{lemma}\label{inter}  Let $F: \R \to \C$ be a smooth, compactly supported function.  Then one has
$$
\left|\int_\R F(y) e(\alpha y)\ dy\right| \leq \frac{\|F^{(k)}\|_{L^1(\R)}}{(2\pi |\alpha|)^k}$$
for any $k \geq 0$ and $\alpha \in \R$.
\end{lemma}

\begin{proof}  The claim is trivial for $k=0$.  For higher $k$, we may assume $\alpha \neq 0$. By integration by parts, we have
$$ \int_\R F(y) e(\alpha y)\ dy = -\frac{1}{2\pi i \alpha} \int_\R F'(y) e(\alpha y)\ dy,$$
and the claim then follows by induction.
\end{proof}

In order to sum the bounds arising from Corollary \ref{googly} (particularly when $k=1$), the following lemma is useful.

\begin{lemma}[Vinogradov-type lemma]\label{vino}  Let $\alpha = \frac{a}{q}+\beta$ for some $\beta = {\mathcal O}^*(1/q^2)$.  Then for any $x < y$, $A, B>0$, and $\theta \in \R/\Z$, we have
$$ \sum_{x < n \leq y} \min\left( A, \frac{B}{|\sin(\pi \alpha n + \theta)|}\right) \leq \left(\left\lfloor \frac{y-x}{q} \right\rfloor + 1\right) (2A + \frac{2}{\pi} B q \log 4q).$$
\end{lemma}

\begin{proof} We may normalise $B=1$.  By subdivision of the interval $[x,y]$ it suffices to show that
$$ \sum_{x < n \leq x+q} \min( A, \frac{1}{|\sin(\pi \alpha n + \theta)|}) \leq 2A + \frac{2}{\pi} q \log 4q$$
for all $x$.  The claim then follows from \cite[Lemma 1]{dr}.  (Strictly speaking, the phase shift $\theta$ is not present in the statement of that lemma, but the proof of the lemma is unchanged with that phase shift.  Alternatively, one can perturb $\alpha$ to be irrational, and then by Kronecker's theorem one can obtain a dense set of phase shifts by shifting the interval $[x,y]$ by an integer, and the claim for general $\theta$ then follows by a limiting argument.)
\end{proof}

Once again, we can save a factor of two by restricting to odd $n$:

\begin{corollary}[Restricting to odd integers]\label{vino-2}  Let $2\alpha = \frac{a}{q}+\beta$ for some $\beta = {\mathcal O}^*(1/q^2)$.  Then for any $x < y$, $A,B>0$, and $\theta \in \R/\Z$, we have
$$ \sum_{x < n \leq y} \min\left( A, \frac{B}{|\sin(\pi \alpha n + \theta)|}\right) \1_{(n,2)=1} \leq \left(\left\lfloor \frac{y-x}{2q} \right\rfloor + 1\right) (2A + \frac{2}{\pi} B q \log 4q).$$
\end{corollary}

\begin{proof}  Writing $n=2m+1$, the expression on the left-hand side is
$$ \sum_{(x-1)/2 < m \leq (y-1)/2} \min\left( A, \frac{B}{|\sin(2 \pi \alpha m + \pi \alpha + \theta)|} \right).$$
The claim then follows from Lemma \ref{vino-2}.
\end{proof}

Now we turn to bilinear estimates.  A key tool here is

\begin{lemma}[Large sieve inequality]  Let $\xi_1,\ldots,\xi_R \in \R/\Z$ be such that $\|\xi_i - \xi_j\|_{\R/\Z} \geq \delta$ for all $1 \leq i < j \leq R$ and some $\delta > 0$.  Let $I = [N_1,N_2]$ be an interval of length $|I| = N_2-N_1 \geq 1$.  Then we have
$$ \sum_{i=1}^R |\sum_{n \in I \cap \Z} a_n e(\xi_i n)|^2 \leq (|I| + \delta^{-1}) \| a_n \|_{\ell^2(\Z)}^2$$
for all complex-valued sequences $(a_n)_{n \in \Z}$.
\end{lemma}

\begin{proof} See \cite[Theorem 3]{mont-survey} (noting that the number of integer points in $I$ is at most $|I|+1$).
\end{proof}

Specialising to $\xi_j$ that are consecutive multiples of $\alpha$ and applying the Cauchy-Schwarz inequality, we conclude

\begin{corollary}[Special case of large sieve inequality]\label{sls}  Let $I, J \subset \R$ be intervals of length at least $1$, and let $\alpha \in \R/\Z$.  Then one has
$$ |\sum_{n \in I \cap \Z} \sum_{m \in J \cap \Z} a_n b_m e( nm \alpha )| \leq \left(|I| + \frac{1}{\inf_{1 \leq j \leq |J|} \|j\alpha\|_{\R/\Z}}\right)^{1/2} \| (a_n)_{n \in \Z} \|_{\ell^2(\Z)} \|(b_m)_{n \in \Z}\|_{\ell^2(\Z)}$$
for all complex-valued sequences $(a_n)_{n \in \Z}, (b_m)_{m \in \Z}$.
\end{corollary}

Again, we can obtain a saving of a factor of $2$ in the main term by restricting $n,m$ to odd numbers:

\begin{corollary}[Restricting to odd numbers]\label{sls-2}  Let $I, J \subset \R$ be intervals of length at least $2$, and let $\alpha \in \R/\Z$.  Then one has
$$ |\sum_{n \in I \cap \Z} \sum_{m \in J \cap \Z} a_n \1_{(n,2)=1} b_m \1_{(m,2)=1} e( nm \alpha )| \leq \left(\tfrac{1}{2} |I| + \frac{1}{\inf_{1 \leq j \leq \tfrac{1}{2} |J|} \|4j\alpha\|_{\R/\Z}}\right)^{1/2} \| (a_n)_{n \in \Z} \|_{\ell^2(\Z)} \|(b_m)_{n \in \Z}\|_{\ell^2(\Z)}$$
for all complex-valued sequences $(a_n)_{n \in \Z}, (b_m)_{m \in \Z}$.
\end{corollary}

\begin{proof}  We can rewrite the left-hand side as
$$ |\sum_{n \in \tfrac{1}{2}I - 1 \cap \Z} \sum_{m \in \tfrac{1}{2} J-1 \cap \Z} \tilde a_n \tilde b_m e(4nm\alpha) |$$
where $\tilde a_n :=  a_{2n+1} e(2 n \alpha)$ and $\tilde b_m := b_{2m+1} e(2m\alpha)$, and the claim then follows from Corollary \ref{sls}.
\end{proof}

The above corollary is useful whenever the $4j\alpha$ for $1 \leq j \leq \tfrac{1}{2} |J|$ stay far away from the origin.  When $J$ is large, this is not always the case, but one can of course rectify this by a subdivision argument:

\begin{corollary}[Subdivision]\label{sls-3}  Let $I, J \subset \R$ be intervals of length at least $2$, and let $\alpha \in \R/\Z$.  Let $M \geq 1$.  Then one has
\begin{align*}
 |\sum_{n \in I \cap \Z} \sum_{m \in J \cap \Z} a_n \1_{(n,2)=1} b_m \1_{(m,2)=1} e( nm \alpha )| &\leq \left(\tfrac{1}{2} |I| + \frac{1}{\inf_{1 \leq j \leq M} \|4j\alpha\|_{\R/\Z}}\right)^{1/2} \\
 &\quad \times \left(\left\lfloor \frac{J}{2M} \right\rfloor + 1\right)^{1/2} \| a_n \|_{\ell^2(\Z)} \|b_m\|_{\ell^2(\Z)}.
\end{align*}
for all complex-valued sequences $(a_n)_{n \in \Z}, (b_m)_{m \in \Z}$.
\end{corollary}

\begin{proof}  We subdivide $J$ into intervals $J_1,\ldots,J_k$ of length $2M$, where $k := \lfloor \frac{J}{2M} \rfloor + 1$.  The claim then follows by applying Corollary \ref{sls-2} to each sub-interval, summing, and using the Cauchy-Schwarz inequality.
\end{proof}

\section{Basic bounds on $S_{\eta,q}(x,\alpha)$}

In all the lemmas in this section, $x \geq 1$ is a real number, $\alpha \in \R/\Z$ is a frequency, $\eta: \R \to \R^+$ is a bounded non-negative measurable function supported in $[0,1]$, and $q \geq 1$ is a natural number.  In some of the lemmas we will also make the additional hypothesis that $\eta$ is smooth, although in applications one can often relax this regularity requirement by a standard limiting argument.  In some cases we will also need $\eta$ to vanish near zero.

The purpose of this section is to collect some basic estimates for manipulating the exponential sums 
\begin{equation}\label{seqxa}
S_{\eta,q}(x,\alpha) := \sum_n \Lambda(n) e(\alpha n) \1_{(n,q)=1} \eta(n/x)
\end{equation}
that already appeared in the introduction. 

We first make the easy observation that $S_{\eta,q}(x,\alpha)$ barely depends on the parameter $q$.

\begin{lemma}\label{eta-smash}  We have
$$ S_{\eta,q}(x,\alpha) = S_{\eta,1}(x,\alpha) + {\mathcal O}^*(\omega(q) \| \eta \|_{L^\infty(\R)} \log x)$$
In particular, if the largest prime factor of $q$ is at most $\sqrt{x}$, then
$$ S_{\eta,q}(x,\alpha) = S_{\eta,1}(x,\alpha) + {\mathcal O}^*( 2.52 \sqrt{x} \| \eta \|_{L^\infty(\R)} )$$
\end{lemma}

In practice, we expect $S_{\eta,q}(x,\alpha)$ to be of size comparable to $x$, and so in practice the error terms here will be utterly negligible in applications, and we will be able to absorb them without any difficulty into a larger error term.

\begin{proof}
We have
$$ |S_{\eta,q}(x,\alpha) - S_{\eta,1}(x,\alpha)| \leq \| \eta \|_{L^\infty(\R)} \sum_{n \leq x: (n,q) > 1} \Lambda(n) .$$
Note that if $\Lambda(n)$ is non-zero and $(n,q)>1$, then $n$ is a power of a prime $p$ dividing $q$, thus
$$ \sum_{n \leq x: (n,q) > 1} \Lambda(n)  =\sum_{p|q} \log p \sum_{j: p^j \leq x} 1.$$
Since $\sum_{j: p^j \leq x} 1 \leq \frac{\log x}{\log p}$, the first claim follows.  For the second claim, observe that
$$ \omega(q) \leq \sum_{p \leq \sqrt{x}} 1 \leq 2.52 \frac{\sqrt{x}}{\log x},$$
where the last inequality follows from \cite[Corollary 1]{rs}.
\end{proof}

Next, we make a simple summation by parts observation that allows one to replace $\eta$ by the sharply truncated cutoff $\1_{[0,1]}$ if desired.

\begin{lemma}\label{etail}  If $\eta$ is smooth, then one has
$$ |S_{\eta,q}(x,\alpha)| \leq \| \eta' \|_{L^1(\R)} \sup_{y \leq x} |S_{\1_{[0,1]}, q}(y,\alpha)|.$$
\end{lemma}

\begin{proof}  Since $\eta(n/x) = - \frac{1}{x} \int_0^x \eta'(y/x) \1_{n \leq y}\ dy$ for all $n \leq x$, we have
$$
 S_{\eta,q}(x,\alpha) = - \frac{1}{x} \int_0^x \eta'(y/x) \sum_n \Lambda(n) e(\alpha n) \1_{n \leq y} \1_{(n,q)=1}\ dy
$$
and thus
$$ |S_{\eta,q}(x,\alpha)| \leq \frac{1}{x} \int_0^x |\eta'(y/x)| |S_{\1_{[0,1]},q}(y,\alpha)|\ dy,$$
and the claim follows.
\end{proof}

We trivially have

\begin{equation}\label{eta-triv}
|S_{\eta,q}(x,\alpha)| \leq S_{\eta,q}(x,0) 
\end{equation}
and we now consider the estimation of the quantity $S_{\eta,q}(x,0)$.

\begin{lemma}\label{rs-lemma}  We have
\begin{equation}\label{sqaw}
S_{\eta,q}(x,0) \leq \| \eta\|_{L^\infty(\R)} S_{\1_{[0,1]},1}(x,0) \leq 1.04 \| \eta\|_{L^\infty(\R)} x.
\end{equation}
If $\eta$ is smooth and supported on $[c,1]$ for some $c>0$ with $cx \geq 10^8$, then
\begin{equation}\label{sqaw-2}
S_{\eta,1}(x,0) = \|\eta\|_{L^1(\R)} x + {\mathcal O}^*\left( \frac{1}{40 \log cx} \|\eta'\|_{L^1(\R)} x \right).
\end{equation}
\end{lemma}

We remark that sharper estimates can be obtained using the machinery from Section \ref{major-sec}, at least when $\eta$ is fairly smooth, though we will not need such improvements here.

\begin{proof}
The first inequality of \eqref{sqaw} is trivial, and the final inequality of \eqref{sqaw} follows from \cite[Theorem 12]{rs} (indeed, one can replace the constant $1.04$ with the slightly smaller $1.03883$).

For \eqref{sqaw-2}, we argue as in Lemma \ref{etail} and write
$$ S_{\eta,1}(x,0) = - \frac{1}{x} \int_{cx}^x \eta'(y/x) \sum_{n \leq y} \Lambda(n)\ dy.$$
From \cite[Theorem 7]{rs-2} (and the hypothesis $cx \geq 10^8$) one has
$$ \sum_{n \leq y} \Lambda(n) = y + {\mathcal O}^*\left( \frac{y}{40 \log cx} \right)$$
and hence
\begin{align*}
S_{\eta,1}(x,0)  &= - \frac{1}{x} \int_{cx}^x \eta'(y/x) y\ dy \\
&\quad + {\mathcal O}^*\left( \frac{1}{x} \frac{1}{40 \log cx} \int_{cx}^x |\eta'(y/x)| y\ dy\right).
\end{align*}
Using the crude bound
$$ \int_{cx}^x |\eta'(y/x)| y\ dy \leq \|\eta'\|_{L^1(\R)} x$$
the claim follows.
\end{proof}

As $\Lambda$ and $\eta$ are real, we have the self-adjointness symmetry
\begin{equation}\label{conj}
S_{\eta,q}(x,-\alpha) = \overline{S_{\eta,q}(x,\alpha)}.
\end{equation}
Also, since $e(n(\alpha+1/2)) = -e(n\alpha)$ when $n$ is odd, we have the anti-symmetry
\begin{equation}\label{asym}
S_{\eta,q}(x,\alpha+1/2) = - S_{\eta,q}(x,\alpha)
\end{equation}
whenever $q$ is even.  More generally, we have the following inequality of Montgomery \cite{montgomery}:

\begin{lemma}[Montgomery's uncertainty principle]\label{mup}  For any $q_0$ dividing $q$, we have
$$
\sum_{a \in \Z/q_0\Z: (a,q_0)=1} \left|S_{\eta,q}(x,\alpha+\frac{a}{q_0})\right|^2 \geq \frac{\mu(q_0)^2}{\phi(q_0)} |S_{\eta,q}(x,\alpha)|^2.$$
\end{lemma}

\begin{proof}  We may of course take $q_0$ to be square-free.  Let $a_n := \Lambda(n) e(\alpha n) \1_{(n,q)=1} \eta(n/x)$, $S(\frac{a}{q_0}) := \sum_n a_n e(an/q_0)$, and $Z := \sum_n a_n$, then the inequality reads
$$ 
\sum_{a \in \Z/q_0\Z: (a,q_0)=1} |S(\frac{a}{q_0})|^2 \geq \frac{1}{\prod_{p|q_0} (p-1)} |Z|^2.$$
But this follows from \cite{montgomery} (particularly equation (10) and the final display in Section 3, and setting $\omega(p)=1$ for all $p|q_0$).  Another proof of this inequality may be found in \cite[Lemma 7.15]{iwaniec}.
\end{proof}

Now we consider $L^2$ estimates on $S_{\eta,q}(x,\alpha)$.  We have a global estimate:

\begin{lemma}[Global $L^2$ estimate]\label{global-l2}  We have
$$
 \int_{\R/\Z} |S_{\eta,q}(x,\alpha)|^2\ d\alpha \leq S_{\eta^2,q}(x,0) \log x.
$$
\end{lemma}
 
\begin{proof}
From the Plancherel identity we have
$$ \int_{\R/\Z} |S_{\eta,q}(x,\alpha)|^2\ d\alpha = \sum_n \eta(n/x)^2 \Lambda(n)^2 \1_{(n,q)=1}.$$
Bounding $\Lambda(n)^2 \leq \Lambda(n) \log x$ on the support of $\eta(n/x)^2$, we obtain the claim.
\end{proof}

We can largely remove\footnote{A version of this inequality was also obtained in an unpublished note of Heath-Brown, which was communicated to me by Harald Helfgott.} the logarithmic factor in the above lemma by restricting to major arcs:

\begin{lemma}[Local $L^2$ estimate]\label{local-l2}  Let $Q,R \geq 1$, and let $\Sigma \subset \R/\Z$ be the set
$$ \Sigma := \bigcup_{q_0 \leq Q} \bigcup_{(a_0,q_0)=1} [\frac{a_0}{q_0}-\frac{1}{2Q^2R^2}, \frac{a_0}{q_0}+\frac{1}{2Q^2R^2}].$$
If $R\sharp|q$, then
\begin{equation}\label{lam}
 \int_\Sigma |S_{\eta,q}(x,\alpha)|^2\ d\alpha \leq \left(\prod_{p \leq Q} \frac{p}{p-1}\right) \frac{\log x}{\log R} S_{\eta^2,q}(x,0).
\end{equation}
\end{lemma}

\begin{proof}  
From Lemma \ref{mup} one has
$$ 
\frac{\mu^2(q_1)}{\phi(q_1)} \int_\Sigma |S_{\eta,q}(x,\alpha)|^2\ d\alpha \leq \sum_{a_1 \in \Z/q_1\Z: (a_1,q_1)=1} \int_{\Sigma + \frac{a_1}{q_1}} |S_{\eta,q}(x,\alpha)|^2\ d\alpha$$
for any $q_1$.  Summing over all $q_1 \leq R$ coprime to $Q\sharp$ and rearranging, we obtain the bound
$$ 
\int_\Sigma |S_{\eta,q}(x,\alpha)|^2\ d\alpha \leq \frac{\sum_{q_1 \leq R: (q_1,Q\sharp)=1} \sum_{a_1 \in \Z/q_1\Z: (a_1,q_1)=1} \int_{\Sigma + \frac{a_1}{q_1}} |S_{\eta,q}(x,\alpha)|^2\ d\alpha}{\sum_{q_1 \leq R:(q_1,Q\sharp)=1} \frac{\mu^2(q_1)}{\phi(q_1)}}.$$
Set
$$ G(R) := \sum_{q_1 \leq R} \frac{\mu^2(q_1)}{\phi(q_1)}.$$
Observe that
$$ G(R) \leq \left(\sum_{q_1 \leq R:(q_1,Q\sharp)=1} \frac{\mu^2(q_1)}{\phi(q_1)}\right) \left(\prod_{p \leq Q} 1 + \frac{1}{\phi(p)}\right)$$
and thus
$$ \frac{1}{\sum_{q_1 \leq R:(q_1,Q\sharp)=1} \frac{\mu^2(q_1)}{\phi(q_1)}} \leq \frac{\prod_{p \leq Q} \frac{p}{p-1}}{G(R)}.$$

Also, from \cite{lint} or \cite[Lemma 3]{montgomery-vaughan} one has $G(R) \geq \log R + 1.07$ for $R \geq 6$, so by direct computation for $1 \leq R \leq 6$ we have $G(R) \geq \log R$ for $R \geq 1$.

To conclude the proof, it thus suffices (in view of Lemma \ref{global-l2}) to show that the sets $\Sigma +\frac{a_1}{q_1}$ are disjoint up to measure zero sets as $a_1,q_1$ vary in the indicated range.  Suppose this is not the case, then $\Sigma + \frac{a_1}{q_1}$ and $\Sigma + \frac{a'_1}{q'_1}$ intersect in a positive measure set for some distinct $a_1,q_1$ and $a'_1,q'_1$ in the indicated range.  Thus one has
$$ |\frac{a_0}{q_0} + \frac{a_1}{q_1} - \frac{a'_0}{q'_0} - \frac{a'_1}{q'_1}| < \frac{1}{Q^2R^2}$$
for some $q_0,q_0 \leq Q$ with $(a_0,q_0)=(a'_0,q'_0)=1$.  The left-hand side is a fraction with denominator at most $Q^2 R^2$, and therefore vanishes.  Since $q_1q'_1$ is coprime with $q_0q'_0$ we conclude that $\frac{a_1}{q_1} = \frac{a'_1}{q'_1}$, contradiction.  The claim follows.
\end{proof}

\begin{remark} As pointed out by the anonymous referee, a slightly stronger version of Lemma \ref{local-l2} (saving a factor of $e^\gamma$ asymptotically) can also be established by modifying the proof of \cite[Theorem 5]{rr}; the main idea is to decompose into Dirichlet characters, as in \cite{bomb}.
\end{remark}

There are several effective bounds on the expression $\prod_{p \leq Q} \frac{p}{p-1}$ appearing in Lemma \ref{local-l2}; see \cite[Theorem 8]{rs}, \cite{dusart-1}, \cite[Theorem 6.12]{dusart-2}.
However, in this paper we will only need to work with the $Q=1$ case, and so we will not use the above lemma here. Specialising Lemma \ref{local-l2} to the case $Q=1$, we conclude

\begin{corollary}\label{uplow}  If $0 < r < 1/2$ and $\sqrt{1/2r}\sharp|q$, one has
$$ \int_{\|\alpha\|_{\R/\Z} \leq r} |S_{\eta,q}(x,\alpha)|^2\ d\alpha \leq \frac{2}{1-\frac{\log(2rx)}{\log x}} S_{\eta^2,q}(x,0).$$
\end{corollary}

We can complement this upper bound with a lower bound:

\begin{proposition}\label{downlow} Let $\eta$ be smooth and $0 \leq r \leq 1/2$.  Then
$$ \int_{\|\alpha\|_{\R/\Z} \leq r} |S_{\eta,q}(x,\alpha)|^2\ d\alpha \geq \frac{(S_{\eta^2,q}(x,0) - \frac{1}{\pi^2 r x} \| \eta' \eta' + \eta \eta''\|_{L^1(\R)} S_{\eta,q}(x,0))_+^2}{ \|\eta\|_{L^2(\R)}^2 x + \| \eta \eta' \|_{L^1(\R)} }.$$
\end{proposition}

Ignoring the error terms, this gives a lower bound of $S_{\eta^2,q}(x,0)$, showing that Corollary \ref{uplow} is essentially sharp up to a factor of $2$ when $rx$ is not too large.

\begin{proof}  From the Parseval formula, one has
$$ S_{\eta^2,q}(x,0) = \int_{\R/\Z} S_{\eta,q}(x,\alpha) F(\alpha)\ d\alpha$$
where $F(\alpha) := \sum_n \eta(n/x) e(-\alpha n)$.  In particular,
$$ \left|\int_{\|\alpha\|_{\R/\Z} \leq r} S_{\eta,q}(x,\alpha) F(\alpha)\ d\alpha\right| \geq S_{\eta^2,q}(x,0) - S_{\eta,q}(x,0) \int_{\|\alpha\|_{\R/\Z} > r} |F(\alpha)|\ d\alpha.$$
and thus by the Cauchy-Schwarz inequality
$$ \int_{\|\alpha\| \leq r} |S_{\eta,q}(x,\alpha)|^2\ d\alpha \geq \frac{(S_{\eta^2,q}(x,0) - S_{\eta,q}(x,0) \int_{\|\alpha\|_{\R/\Z} > r} |F(\alpha)|\ d\alpha)^2_+}{\int_{\R/\Z} |F(\alpha)|^2\ d\alpha}.$$
By the Plancherel theorem, one has
$$ \int_{\R/\Z} |F(\alpha)|^2\ d\alpha = \sum_n \eta(n/x)^2$$
and hence by \eqref{f0}
$$ \int_{\R/\Z} |F(\alpha)|^2\ d\alpha = \|\eta\|_{L^2(\R)}^2 x + {\mathcal O}^*( \| \eta \eta' \|_{L^1(\R)} ).$$
Similarly, from \eqref{fk} (with $k=2$) one has
$$ |F(\alpha)| \leq 
\frac{1}{2x|\sin(\pi \alpha)|^2} \| \eta' \eta' + \eta \eta''\|_{L^1(\R)}$$
for any $\alpha$ and thus
$$ \int_{\|\alpha\|_{\R/\Z} \geq r} |F(\alpha)|\ d\alpha \leq 
\frac{\cot(\pi r)}{\pi x} \| \eta' \eta' + \eta \eta''\|_{L^1(\R)}.$$
Bounding
$$ \cot(\pi r) \leq \frac{1}{\pi r}$$
the claim follows.
\end{proof}

We can clean up the error terms as follows:

\begin{corollary}\label{downlow-2} Let $\eta$ be smooth and supported on $[c,1]$ for some $c>0$, and suppose that $\frac{1}{2x} \leq r \leq 1/2$ and $q = \sqrt{x}\sharp$.  We normalise $\|\eta\|_{L^2(\R)} = 1$.  Assume furthermore that
\begin{align}
cx &\geq 10^8 \label{c8} \\
x &\geq 10^4 \|\eta \eta'\|_{L^1(\R)} \label{neat}\\
\log(cx) &\geq 5 \|\eta \eta'\|_{L^1(\R)} \label{alamo}\\
x &\geq 10^8 \|\eta\|_{L^\infty(\R)}^4 \label{10q}.\\
r x &\geq 20 \| \eta' \eta' + \eta \eta''\|_{L^1(\R)} \| \eta\|_{L^\infty(\R)} \label{r0b}
\end{align}
Then one has
\begin{equation}\label{hog}
 S_{\eta^2,q}(x,0) = x (1 + {\mathcal O}^*( 0.02 ))
\end{equation}
and
\begin{equation}\label{hog-2} 
\int_{\|\alpha\| \leq r} |S_{\eta,q}(x,\alpha)|^2\ d\alpha \geq 0.94 x 
\end{equation}
\end{corollary}

\begin{proof}  
From Lemma \ref{rs-lemma} and \eqref{c8}, \eqref{alamo} one has
$$ S_{\eta^2,1}(x,0) = x (1 + {\mathcal O}^*( 0.01 ))$$
and by Lemma \ref{eta-smash} and \eqref{10q} we conclude \eqref{hog}.

Let us denote the quantity $\int_{\|\alpha\| \leq r} |S_{\eta,q}(x,\alpha)|^2\ d\alpha$ by $A$.  By Proposition \ref{downlow}, \eqref{neat} and Lemma \ref{rs-lemma} we have
$$ A \geq 0.999 \frac{1}{\|\eta\|_{L^2(\R)}^2 x} \left(S_{\eta^2,q}(x,0) - \frac{1.04}{\pi^2 r} \| \eta' \eta' + \eta \eta''\|_{L^1(\R)} \| \eta\|_{L^\infty(\R)} \right)_+^2.$$
By \eqref{r0b}, \eqref{hog} we have
$$ \frac{1.04}{\pi^2 r} \| \eta' \eta' + \eta \eta''\|_{L^1(\R)} \| \eta\|_{L^\infty(\R)} \leq 0.01 S_{\eta^2,q}(x,0)$$
and so
$$ A \geq 0.97 \frac{1}{x} S_{\eta^2,q}(x,0)^2$$
and thus by \eqref{hog}
$$ A \geq 0.94 x,$$
which is \eqref{hog-2}.
\end{proof}

The estimate in Corollary \ref{uplow} is quite sharp when $r$ is small (of size close to $1/x$), but not when $r$ is large.  For this, we have an alternate estimate:

\begin{proposition}[Mesoscopic $L^2$ estimate]\label{meso}  Suppose that $q = \sqrt{x}\sharp$.  Let $H \geq 10^2$, and write $D_H(\alpha) := \sum_{h=1}^H e(h \alpha)$ for the Dirichlet-type kernel.  Then we have
$$
\int_{\R/\Z} |S_{\eta,q}(x,\alpha)|^2 |D_H(\alpha)|^2\ d\alpha \leq (1+\eps) \times 8 H^2 x \|\eta\|_{L^\infty(\R)}^2
$$
where $\eps$ is the quantity
\begin{equation}\label{error}
 \eps := \frac{0.13 \log x}{H} + \frac{(e^\gamma \log \log(2H) + \frac{2.507}{\log\log(2H)}) \log(9H)}{2H}.
\end{equation}
\end{proposition}

By applying this proposition with $H$ slightly larger than $\log x$ and using standard lower bounds on $|D_H(\alpha)|$, we conclude that
$$ \int_{\|\alpha\| = o(\frac{1}{\log x})} |S_{\eta,q}(x,\alpha)|^2\ d\alpha \leq (8+o(1)) x \|\eta\|_{L^\infty(\R)}^2.$$
In comparison, Corollary \ref{uplow} gives an upper bound of $(2+o(1)) \frac{\log x}{\log\log x} x \|\eta\|_{L^2(\R)}^2$ for this integral, while Proposition \ref{downlow} gives a lower bound of $(1-o(1)) x \|\eta\|_{L^2(\R)}^2$ (if $\eta$ is smooth).  Thus we see that the bound in Proposition \ref{meso} is only off by a factor of $8$ or so, if $\eta$ is close to $\1_{[0,1]}$.  It seems of interest to find efficient variants of this proposition in which the $|D_H(\alpha)|^2$ weight is replaced by a weight concentrated on multiple major arcs, as in Lemma \ref{local-l2}, as this may be of use in further work on Goldbach-type problems.

\begin{proof}  We may normalise $\|\eta\|_{L^\infty(\R)} = 1$.
By the Plancherel identity, we may write the left-hand side as
$$ \sum_{h=1}^H \sum_{h'=1}^H \sum_n \Lambda(n) \eta(n/x) \1_{(n,q)=1} 
\Lambda(n+h'-h) \eta((n+h'-h)/x) \1_{(n+h'-h,q)=1}.$$
The diagonal contribution $h=h'$ can be bounded by
$$ H \sum_n \Lambda(n) \eta(n/x) \log x$$
which by Lemma \ref{rs-lemma} is bounded by $1.04 H x \log x$.  Now we consider the off-diagonal contribution $h \neq h'$.  As $\Lambda(n) \1_{(n,q)=1}$ is supported on the odd primes $p$ less than or equal to $x$, we restrict attention to the contribution when $h-h'$ is even, and can bound this contribution by
$$
 \sum_{h'=1}^H \sum_{1 \leq h \leq H: h \neq h'; 2|h-h'} \log^2 x | \{ p \leq x: p+h-h' \hbox{ prime}, p+h-h' \leq x\}|,
$$
which by the pigeonhole principle is bounded by
\begin{equation}\label{aha}
H \sum_{1 \leq h \leq H: h \neq h'; 2|h-h'} \log^2 x | \{ p \leq x: p+h-h' \hbox{ prime}, p+h-h' \leq x\}|,
\end{equation}
for some $1 \leq h' \leq H$, which we now fix. Applying the main result of Siebert \cite{sie}, we have
$$ 
| \{ p \leq x: p+h-h' \hbox{ prime}, p+h-h' \leq x\}| \leq 8 {\mathfrak S}i_2 \frac{x}{\log^2 x} \prod_{p|h-h'; p>2} \frac{p-1}{p-2}$$
where ${\mathfrak S}_2 := 2 \prod_{p>2} (1-\frac{1}{(p-1)^2})$ is the twin prime constant.  We can thus bound \eqref{aha} by
$$
8 {\mathfrak S}_2 H x \sum_{1 \leq h \leq H: h \neq h'; 2|h-h'} \prod_{p|h-h'; p>2} \frac{p-1}{p-2}.$$
If we let $f$ be the multiplicative function
$$ f(n) := \1_{(n,2)=1} \mu^2(n) \prod_{p|n: p>2} \frac{1}{p-2}$$
then we have
$$ \prod_{p|h-h'; p>2} \frac{p-1}{p-2} = \sum_{1 \leq n \leq H; (n,2)=1: n|h-h'} f(n)$$
whenever $1 \leq h,h' \leq H$ with $h \neq h'$, so we may bound the preceding expression by
$$ 8 {\mathfrak S}_2 H x \sum_{1 \leq n \leq H: (n,2)=1} f(n) \sum_{1 \leq h \leq H: 2n | h-h'} 1.$$
We may bound 
\begin{equation}\label{divert}
 \sum_{1 \leq h \leq H: 2n | h-h'} 1 \leq \frac{H}{2n} + 1
\end{equation}
so that \eqref{aha} is then bounded by
$$ 8 {\mathfrak S}_2 H x \left(H \sum_{n=1}^\infty \frac{f(n)}{2n} + \sum_{1 \leq n \leq H} f(n)\right).$$
By evaluating the Euler product one sees that
$$ \sum_{n=1}^\infty \frac{f(n)}{2n} = \tfrac{1}{2} \prod_p (1+f(p)) = \frac{1}{{\mathfrak S}_2}.$$
To evaluate the $f$ summation, we observe that\footnote{Alternatively, one can sum $f$ directly by using effective bounds on sums of multiplicative functions, as in \cite{ramare}.  This will save a factor of $\log\log H$ in the upper bounds, but in our applications this loss is quite manageable.}
$$ f(n) = \1_{(n,2)=1} \frac{\mu^2(n)}{\phi(n)} \frac{2n}{\phi(2n)} \tfrac{1}{2} \prod_{p|n: p>2} (1-\frac{1}{(p-1)^2})^{-1}.$$
We can bound
$$ \tfrac{1}{2} \prod_{p|n > 2} (1-\frac{1}{(p-1)^2})^{-1} \leq \frac{1}{{\mathfrak S}_2}$$
and from \cite[Theorem 15]{rs} one has
$$ \frac{2n}{\phi(2n)} \leq e^\gamma \log \log(2n) + \frac{2.507}{\log\log(2n)}.$$
Since $n \leq H$ and $H \geq 10^2$, we conclude that
$$ \frac{2n}{\phi(2n)} \leq e^\gamma \log \log(2H) + \frac{2.507}{\log\log(2H)}$$
for any $1 \leq n \leq H$ with $(n,2)=1$ (the cases when $n$ is so small that $e^\gamma \log \log(2n) + \frac{2.507}{\log\log(2n)}$ can exceed $e^\gamma \log \log(2H) + \frac{2.507}{\log\log(2H)}$, and specifically when $n = 1,3,5$, can be verified by hand).  Thus we may bound
$$ {\mathfrak S}_2 \sum_{1 \leq n \leq H} f(n) \leq \left(e^\gamma \log \log(2H) + \frac{2.507}{\log\log(2H)}\right) \sum_{n \leq H: (n,2)=1} \frac{\mu^2(n)}{\phi(n)}.$$
Since $\phi(n)=\phi(2n)$ when $n$ is odd, we have
$$ \sum_{n \leq H: (n,2)=1} \frac{\mu^2(n)}{\phi(n)} \leq \tfrac{1}{2} \sum_{n \leq 2H} \frac{\mu^2(n)}{\phi(n)}$$
and hence by \cite[Lemma 3.5]{ramare}, one has
$$ \sum_{n \leq H: (n,2)=1} \frac{\mu^2(n)}{\phi(n)} \leq \tfrac{1}{2}( \log(2H) + 1.4709 ) \leq \tfrac{1}{2} \log(9H).$$
Combining all these estimates we obtain the claim.
\end{proof}

\begin{remark} As observed by the anonymous referee, when $H$ is an integer, the second term in \eqref{error} may be deleted by using \cite[Lemma 5.1]{eigen} as a substitute for \eqref{divert} (and retaining the sum over $h'$, rather than working only with the worst-case $h'$).
\end{remark}

To deal with $S_{\eta,q}(x,\alpha)$ on minor arcs, we follow the standard approach of Vinogradov by decomposing this expression into linear (or ``Type I'') and bilinear (or ``Type II'') sums.  We will take advantage of the following variant of Vaughan's identity \cite{vaughan}, which we formulate as follows:

\begin{lemma}[A variant of Vaughan's identity]\label{vaughan-lemma}  Let $U,V \geq 1$.  Then for any function $F: \Z \to \C$ supported on the interval $(V, UV^2)$, one has
$$ |\sum_n \Lambda(n) F(n)| \leq T_I + T_{II}$$
where $T_I$ is the Type I sum\footnote{Strictly speaking, $T_I$ is an average of Type I sums, rather than a single Type I sum; however we shall abuse notation and informally refer to $T_I$ as a Type I sum.}
$$ T_I := \sum_{d \leq UV} |\sum_n (\log n + c_d \log d) F(dn)|$$
for some complex coefficients $c_d$ (depending on $F$) with $|c_d| \leq 1$, and $T_{II}$ is the Type II sum
$$ T_{II} := |\sum_{d > U} \sum_{w > V} \mu(d) g(w) F(dw)|$$
where $g(w)$ is the function
$$ g(w) := \sum_{b|w: b > V} \Lambda(b) - \tfrac{1}{2} \log w.$$
\end{lemma}

This differs slightly from the standard formulation of Vaughan's identity, as we have subtracted a factor of $\tfrac{1}{2} \log w$ from the coefficient $g(w)$ of the Type II sum.  This has the effect of improving the Type II sum by a factor of two, with only a negligible cost to the (less important) Type I term.

\begin{proof}  We split $\mu = \mu \1_{\leq U} + \mu \1_{>U}$ and $\Lambda = \Lambda \1_{\leq V} + \Lambda \1_{>V}$, where $\1_{\leq U}(n) := \1_{n \leq U}$, and similarly for $\1_{>U}$, $\1_{\leq V}$, $\1_{>V}$.  We then have
\begin{equation}\label{vaughan}
\begin{split}
\Lambda &= \mu \ast \Lambda \ast 1 \\
&= \mu \1_{\leq U} \ast \Lambda \ast 1 - \mu \1_{\leq U} \ast \Lambda \1_{\leq V} \ast 1 + \mu \1_{>U} \ast \Lambda \1_{>V} \ast 1 + \mu \ast \Lambda \1_{\leq V} \ast 1\\
&= \mu \1_{\leq U} \ast \log - \mu \1_{\leq U} \ast \Lambda \1_{\leq V} \ast 1 + \mu \1_{>U} \ast \Lambda \1_{>V} \ast 1  + \Lambda \1_{\leq V}
\end{split}
\end{equation}
We sum this against $F$, noting that $F$ vanishes on the support of the final term $\Lambda \1_{\leq V}$, to conclude that
\begin{align*}
\sum_n \Lambda(n) F(n) &= 
\sum_{d \leq U} \mu(d) \sum_{n} (\log n) F(dn) \\
&\quad - \sum_{d \leq UV} f(d) \sum_{n} F(dn) \\
&\quad + \sum_{d > U} \sum_{w > V} \mu(d) (g(w)+\tfrac{1}{2} \log w) F(dw)
\end{align*}
where
$$ f(d) := \sum_{b|d: d/U \leq b \leq V} \mu(\frac{d}{b}) \Lambda(b).$$
Observe that when $d>U$ and $w>V$, then the term $\mu(d) (\tfrac{1}{2} \log w) F(dw)$ vanishes unless $U < d \leq UV$.  We may thus bound
\begin{align*}
|\sum_n \Lambda(n) F(n)| &= 
\sum_{d \leq UV} |\sum_{n} (\log n) F(dn)| \\
&\quad + \sum_{d \leq UV} |f(d)| |\sum_{n} F(dn)| \\
&\quad + |\sum_{d > U} \sum_{w > V} \mu(d) g(w) F(dw)|.
\end{align*}
Note that
$$ |f(d)| \leq \sum_{b|d} \Lambda(b) = \log d.$$
Since
$$ |\sum_{n} (\log n) F(dn)| + |\sum_{n} F(dn)| \log d  = |\sum_n (\log n + c_d \log d) F(dn)|$$
for some complex constant $c_d$ with $|c_d|=1$, we obtain the claim.
\end{proof}

Note that as
$$ 0 \leq \sum_{b|w: b > V} \Lambda(b) \leq \sum_{b|w} \Lambda(b) = \log w$$
we have the bound
\begin{equation}\label{g-bound}
|g(w)| \leq \tfrac{1}{2} \log w.
\end{equation}

\section{Minor arcs}

We now prove the following exponential sum estimate, which will then be used in the next section to derive Theorem \ref{expsum}:

\begin{theorem}[Bound for minor arc sums]\label{strong-minor}  Let $4\alpha = \frac{a}{q} + \beta$ for some natural number $q \geq 4$ with $(a,q)=1$ and some $\beta = {\mathcal O}^*(1/q^2)$. Let $1 < U,V < x$, and suppose that we have the hypotheses
\begin{align}
U V &\leq x/4 \label{v-bound}\\
UV^2 &\geq x \label{uv-big}\\
U, V &\geq 40. \label{v-bound-5}
\end{align}
Then one has
\begin{align}
|S_{\eta,2}(x,\alpha)| &\leq  
0.5 \frac{x}{q} (\log x) \log\left(\frac{2UV}{q} + 4\right) + 0.89 \left(UV+\frac{5}{2} q\right) (8+\log q) \log(2x) \label{term-1} \\
&\quad + \left(0.1 \frac{x}{\sqrt{q}} + 0.39 \frac{x}{\sqrt{x/q}} \right) (\log \frac{x}{UV}) \log \frac{V x}{U} \\
&\quad + \left(0.55 \frac{x}{\sqrt{U}} + 0.78 \frac{x}{\sqrt{V}} \right) \log \frac{x}{U}.
\end{align}
Furthermore, if $a= \pm 1$ and $UV < q-1$, then we may replace the term \eqref{term-1} in the above estimate with
\begin{equation}\label{term-1-alt}
 \frac{96}{\pi^2} \frac{x}{(x/q)^2} \log(4x) \log \frac{4e q}{\pi}.
\end{equation}
\end{theorem}

We now prove Theorem \ref{strong-minor}.  By Lemma \ref{vaughan-lemma}, we have
$$
|S_{\eta_0,2}(x,1)| \leq T_I + T_{II}
$$ 
where $T_I$ is the Type I sum
$$
T_I := \sum_{d \leq UV: (d,2)=1} 
|\sum_{n: (n,2)=1} (\log n + c_d \log d) e(\alpha dn) \eta_0(dn/x)|
$$
and $T_{II}$ is the Type II sum
$$ T_{II} := |\sum_{d > U: (d,2)=1} \sum_{w > V: (w,2)=1} \mu(d) g(w) e(\alpha dw) \eta_0(dw/x)|.$$

\subsection{Estimation of the Type I sum}

We first bound the Type I sum $T_I$.  We begin by estimating a single summand
\begin{equation}\label{node}
|\sum_{n} (\log n +c_d \log d) e(dn\alpha) \eta_0(dn/x) \1_{(n,2)=1}|.
\end{equation}
By Corollary \ref{googly}, we may bound this expression by
$$ \min\left( \tfrac{1}{2} \|F\|_{L^1(\R)} + \tfrac{1}{2} \|F'\|_{L^1(\R)},
\frac{\|F'\|_{L^1(\R)}}{2|\sin(2\pi d\alpha)|}, \frac{\|F''\|_{L_1(\R)}}{2|\sin(2\pi d\alpha)|^2} \right)$$
where $F(y) = F_d(y) := \eta_0(dy/x) (\log y + c_d \log d)$.  We have
$$ F'(y) = \frac{d}{x} \eta'_0(dy/x) (\log y + c_d \log d) + \frac{\eta_0(dy/x)}{y} $$
and
$$ F''(y) = \frac{d^2}{x^2} \eta''_0(dy/x) (\log y + c_d \log d) + 2 \frac{d}{x} \frac{\eta'_0(dy/x)}{y} - \frac{\eta_0(dy/x)}{y^2}.$$
Strictly speaking, $F$ is not infinitely smooth, and so Corollary \ref{googly} cannot be directly applied; however, we may perform the standard procedure of mollifying $\eta_0$ (and thus $F$) by an infinitesimal amount in order to make all derivatives here well-defined, and then performing a limiting argument.  Rather than present this routine argument explicitly, we shall abuse notation and assume that $\eta_0$ has been infinitesimally mollified (alternatively, one can interpret the derivatives here in the sense of distributions, and replace the $L^1$ norm by the total variation norm in the event that the expression inside the norm becomes a signed measure instead of an absolutely integrable function).

Since $d \leq U$ and $\eta$ is supported on $[1/4,1]$, $\eta(dy/x)$ is supported on $[x/4d,x/d]$.  In particular, $|\log y + c_d \log d| \leq \log x$.  We thus have
\begin{align*}
 \|F\|_{L^1(\R)} &\leq \| \eta_0 \|_{L^1(\R)} \frac{x}{d} \log x \\
 \|F'\|_{L^1(\R)} &\leq \| \eta'_0 \|_{L^1(\R)} \log x + \|\eta_0\|_{L^\infty(\R)} \log 4 \\
 \|F''\|_{L^1(\R)} &\leq \| \eta''_0\|_{L^1(\R)} \frac{d}{x} \log x + 2 \|\eta'_0\|_{L^\infty(\R)} \frac{d}{x} \log 4 + \| \eta_0\|_{L^\infty(\R)} \frac{4d}{x}.
\end{align*}
Routine computations using \eqref{Eta0-def} show that
\begin{align}
\| \eta_0 \|_{L^1(\R)} &= 1 \label{s1}\\
\| \eta_0 \|_{L^\infty(\R)} &= 4 \log 2 \label{s1-2}\\
\| \eta'_0 \|_{L^1(\R)} &= 8 \log 2 \label{s1-3}\\
\| \eta'_0 \|_{L^\infty(\R)} &= 16 \label{s1-4}\\
\| \eta''_0 \|_{L^1(\R)} &= 48.\label{s1-5} 
\end{align}
Note that $\| \eta'_0\|_{L^1(\R)}$ and $\|\eta''_0\|_{L^1(\R)}$ can also be interpreted as the total variation of the functions $\eta_0$ and $\eta'_0$ respectively, which may be an easier computation (especially since $\eta''_0$ is not a function before mollification, but is merely a signed measure).
Inserting these bounds, we conclude that
\begin{align*}
 \|F\|_{L^1(\R)} &\leq \frac{x}{d} \log x \\
 \|F'\|_{L^1(\R)} &\leq 8 (\log 2) \log 2x \\
 \|F''\|_{L^1(\R)} &\leq 48 \frac{d}{x} \log 4x.
\end{align*}
We conclude that
\begin{equation}\label{amble}
T_I \leq \sum_{d \leq UV} \1_{(d,2)=1} \min\left( \tfrac{1}{2} \frac{x}{d} \log x + 4 (\log 2) \log 2x, \frac{4 (\log 2) \log 2x}{|\sin(2\pi d\alpha)|}, \frac{d}{x} \frac{24 \log 4x}{|\sin(2\pi d\alpha)|^2}\right).
\end{equation}

We first control the contribution to \eqref{amble} when $d \leq q/2$.  In this case, $d$ is not divisible by $q$; as $(a,q)=1$, this implies that $ad$ is not divisible by $q$ either.  We therefore have
\begin{equation}\label{ala}
 \| 4d\alpha \|_{\R/\Z} \geq \|ad/q\|_{\R/\Z} - d|\beta| \geq \frac{1}{q} - \frac{q/2}{q^2} = \frac{1}{2q}
\end{equation}
and thus
\begin{equation}\label{daa}
 \frac{1}{|\sin(2\pi d\alpha)|} \leq \frac{1}{|\sin(\pi/4q)|} \leq 2q
\end{equation}
thanks to \eqref{sinx}.   Thus contribution here \eqref{amble} of the $d \leq q/2$ terms are bounded by
$$ 4 (\log 2) \log 2x \sum_{d \leq q/2} \1_{(d,2)} \min( 2q, \frac{1}{|\sin(2\pi d\alpha)|} ).$$
By Corollary \ref{vino-2} one has
$$ \sum_{d \in \Z: -q/2 \leq d \leq q/2} \1_{(d,2)} \min( 2q, \frac{1}{|\sin(2\pi d\alpha)|} ) \leq \frac{2}{\pi} q\log 4q  + 4 q$$
so by symmetry we may thus bound the contribution of the $d \leq q/2$ terms to \eqref{amble} by
$$ 2 \log 2 \log 2x (\frac{2}{\pi} q\log 4q  + 4 q).$$
Now consider the contribution to \eqref{amble} of a block of the form $2jq+\frac{q}{2} < d \leq 2(j+1)q + \frac{q}{2}$ for a natural number $j$ with $j \leq \frac{UV}{2q} - \frac{1}{4}$.  This contribution can be bounded by
$$ \sum_{2jq+\frac{q}{2} < d \leq 2(j+1)q + \frac{q}{2}} 
\1_{(d,2)=1} \min\left( \tfrac{1}{2} \frac{x}{2jq+\frac{q}{2}} \log x + 4 (\log 2) \log 2x, \frac{4 (\log 2) \log 2x}{|\sin(2\pi d\alpha)|} \right),$$
which by Corollary \ref{vino-2} is bounded by
$$ 
\frac{x}{2jq+\frac{q}{2}} \log x + 8 (\log 2) \log 2x + 4 (\log 2) \log 2x \frac{2}{\pi} q \log 4q$$
which we can crudely bound by
$$ 
\frac{x}{2jq+\frac{q}{2}} \log x + 4 (\log 2) \log 2x (\frac{2}{\pi} q \log 4q + 4q).$$
Combining all the contributions to \eqref{amble}, we obtain the bound
\begin{align*}
T_I &\leq \sum_{0 \leq j \leq \frac{UV}{2q}-\frac{1}{4}} \frac{x}{2jq+\frac{q}{2}} \log x \\
&\quad + (\frac{UV}{2q}+\frac{5}{4}) (\frac{2}{\pi} q \log 4q + 4q ) \times 4 (\log 2) \log 2x.
\end{align*}
By the integral test, one has
\begin{align*}
 \sum_{0 \leq j \leq \frac{UV}{2q}-\frac{1}{4}} (\frac{x}{2jq+\frac{q}{2}}) 
&\leq \frac{1}{2q} \int_{q/2}^{UV+2q} \frac{x}{y}\ dy \\
&= \frac{x}{2q} \log(\frac{2UV}{q} + 4)
\end{align*}
and so
$$ T_I \leq \frac{x}{2q} \log(\frac{2UV}{q} + 4) \log x + (\frac{UV}{2q}+\frac{5}{4}) (\frac{2}{\pi} q \log 4q + 4q) \times 4 (\log 2) \log 2x.$$
We can write
$$ \frac{2}{\pi} q \log 4q + 4q \leq \frac{2}{\pi} q (8 + \log q) $$
and
$$ \tfrac{1}{2} \times \frac{2}{\pi} \times 4 \log 2 \leq 0.89$$
and so
\begin{equation}\label{ti-p}
 T_I \leq 0.5 \frac{x}{q} \log(\frac{2UV}{q}+4) \log x + 0.89 (UV+\frac{5}{2} q) (8 + \log q) \log x
\end{equation}
which gives the term \eqref{term-1}.

Now suppose that $a=\pm 1$ and $UV < q-1$.  By the symmetry \eqref{conj} we may take $a=1$, thus $\alpha = \frac{1}{4q} + {\mathcal O}^*( \frac{1}{4q^2}) \leq \frac{1}{4(q-1)}$.  In particular, for $d \leq UV$ one has $d \leq q-2$ and therefore $0 < 2\pi d\alpha < \pi/2$.  In particular, we have
$$ \sin( 2\pi d \alpha ) \geq \sin\left( \frac{2 \pi d}{4(q-1)} \right).$$
From \eqref{amble}, we conclude that
$$ T_I \leq \frac{24 \log 4x}{x} \sum_{d=1}^{q-2} d\ \cosec^2 \frac{\pi d}{2(q-1)}.$$
The function $d \mapsto d\ \cosec^2 \frac{\pi d}{2(q-1)}$ is convex on $[0,\pi]$ (indeed, both factors are already convex), and so by the trapezoid rule
$$ T_I \leq \frac{24 \log 4x}{x} \int_{1/2}^{q-3/2} y \cosec^2 \frac{\pi y}{2(q-1)}\ dy.$$
Using the anti-derivative
$$ \int y \cosec^2 y\ dy = \log|\sin y| - y \cot y$$
we can evaluate the right-hand side as
$$ (\frac{2(q-1)}{\pi})^2 \frac{24 \log 4x}{x} (\log|\sin y| - y \cot y)|^{y=\pi/2-\pi/4(q-1)}_{y=\pi/4(q-1)}$$
The function $y \cot y$ varies between $0$ and $1$ on $[0,\pi/2]$, and thus
$$
 T_I \leq (\frac{2(q-1)}{\pi})^2 \frac{24 \log 4x}{x} (\log \cot \frac{\pi}{4(q-1)} + 1)
$$
which, on bounding $\cot \frac{\pi}{4(q-1)} \leq \frac{4(q-1)}{\pi} \leq \frac{4q}{\pi}$ and $\frac{2(q-1)}{\pi} \leq \frac{2q}{\pi}$, gives 
\begin{equation}\label{ti-p2}
T_I \leq \frac{96}{\pi^2} \frac{x}{(x/q)^2} \log(4x) \log(\frac{4eq}{\pi})
\end{equation} 
which is the alternate contribution \eqref{term-1-alt}.

\subsection{Estimation of the Type II sum}

We now control $T_{II}$.  From \eqref{eta0} and the triangle inequality, we thus have
\begin{equation}\label{tii}
T_{II} \leq 4 \int_0^\infty F(W) \frac{dW}{W}
\end{equation}
where
$$ F(W) :=  |\sum_{d > U} \sum_{w>V} \mu(d) \1_{(d,2)=1} \1_{[x/2W,x/W]}(d) g(w) \1_{(w,2)=1} \1_{[W/2,W]}(w) e( \alpha dw )|.$$
Observe that $F(W)$ vanishes unless 
\begin{equation}\label{Regime}
V \leq W \leq \frac{x}{U},
\end{equation}
and we henceforth restrict attention to this regime.

From \eqref{Regime}, \eqref{v-bound-5} we see that the intervals $[x/2W,x/W]$, $[W/2,W]$ both have length at least $2$.
We now apply Corollary \ref{sls-3} with $M := q/2$ and use \eqref{g-bound} to bound
\begin{equation}\label{del1}
 F(W) \leq \tfrac{1}{2} \left(\tfrac{1}{2} \frac{W}{2} + \frac{1}{\delta}\right)^{1/2} \left(\lfloor \frac{x}{2Wq} \rfloor + 1\right)^{1/2} A^{1/2} B^{1/2}
\end{equation}
where\footnote{One could achieve a very slight improvement to the $A$ factor by exploiting the fact that $\mu^2(n)$ vanishes when $n$ is divisible by an odd square, but we will not do so here to keep the exposition simple.}
\begin{align*}
\delta &:= \inf_{1 \leq j \leq q/2} \| 4 j \alpha\|_{\R/\Z} \\
A &:= \sum_{d \in [x/2W,x/W]} \1_{(d,2)=1} \\
B &:= \sum_{w \in [W/2,W]} \1_{(w,2)=1} \log^2 w.
\end{align*}
From the computation \eqref{ala} we have
\begin{equation}\label{del2}
 \delta \geq \frac{1}{2q}.
\end{equation}
Now we bound $A$ and $B$.  Clearly we have
$$ |A| \leq \frac{x}{4W} + 1.$$
Similarly, for $B$ we bound $\log^2 w$ by $\log^2 W$, and conclude that
$$ |B| \leq (\frac{W}{4} + 1) \log^2 W.$$
By \eqref{v-bound-5}, \eqref{Regime} we thus have
\begin{align*}
 |A| &\leq \frac{1.1}{4} \frac{x}{W} \\
 |B| &\leq \frac{1.1}{4} W \log^2 W
\end{align*}
and thus by \eqref{del1}, \eqref{del2}
$$ F(W) \leq \frac{1.1}{8} (\frac{W}{4} + 2q)^{1/2} (\frac{x}{2Wq} + 1)^{1/2} x^{1/2}  \log W.$$
Crudely bounding\footnote{By avoiding this step, one could save a modest amount in the numerical constants in the estimates, but at the cost of a more complicated argument.} $(a+b)^{1/2} \leq a^{1/2}+b^{1/2}$ we thus have
$$
F(W) \leq \frac{1.1}{8} (\frac{1}{2\sqrt{2}} \frac{x}{\sqrt{q}} +
\tfrac{1}{2} \sqrt{xW} + \frac{1}{\sqrt{2}} \frac{x}{\sqrt{W}} + \sqrt{2} \frac{x}{\sqrt{x/q}} ) \log W.$$
We integrate this expression over \eqref{Regime} against the measure $\frac{4dW}{W}$ to bound \eqref{tii}.  To simplify the computations slightly at the cost of a slight degradation of the numerical constants, we bound $\log W$ crudely by $\log \frac{x}{U}$ for the middle two terms for $F(W)$.  Since
$$  4 \int_{V \leq W\leq\frac{x}{U}} \frac{dW}{W} = 2 \log \frac{x}{UV} \log \frac{V x}{U},$$
we thus obtain the bound
\begin{align*}
T_{II} &\leq \frac{1.1}{4} \left(\frac{1}{2\sqrt{2}} \frac{x}{\sqrt{q}} + \sqrt{2} \frac{x}{\sqrt{x/q}} \right) \log \frac{x}{UV} \log \frac{V x}{U} \\
&\quad + 1.1 \left( \tfrac{1}{2} \frac{x}{\sqrt{U}} + \frac{1}{\sqrt{2}} \frac{x}{\sqrt{V}} \right) \log W.
\end{align*}
Combining this with \eqref{ti-p} and \eqref{ti-p2} we obtain the claim.

\begin{remark} When $\beta$ is very small (of size $O(1/x)$ or so), one can eliminate most of the $0.5 \frac{x}{q} \log x \log(\frac{2UV}{q} + 4)$ term in \eqref{term-1} by working with $S_{\eta,q}(x,\alpha) - \frac{\mu^2(q)}{\phi(q)} S_{\eta,q}(x,\beta)$ instead of $S_{\eta,2}(x,\alpha)$, which effectively replaces the phase $e(\alpha n)$ with the variant phase $(e(an/q) - \frac{\mu^2(q)}{\phi(q)}) e(\beta n) \1_{(n,q)=1}$.  The main purpose of this replacement is to delete the non-cancellative component of the Type I sum (which, for small values of $\beta$, occurs when $d$ is divisible by $q$), which would otherwise generate the $\frac{x}{q} \log^2 x$ type term in \eqref{Sax-2}.  Such an improvement may be useful in subsequent work on Goldbach-type problems, but we will not pursue it further here\footnote{For very small values of $q$, it is likely that one should instead proceed using different identities than Vaughan-type identities, for instance by performing an expansion into Dirichlet characters instead.}.
\end{remark}

\section{Proof of Theorem \ref{expsum}}
	
In this section we use Theorem \ref{strong-minor} to establish Theorem \ref{expsum}, basically by applying specific values of $U$ and $V$.  With more effort, one could optimise in $U$ and $V$ more carefully than is done below, which would improve the numerical values in Theorem \ref{expsum} by a modest amount, but we will not do so here to keep the exposition as simple as possible.  

We begin with \eqref{Sax}.  We apply Theorem \ref{strong-minor} with 
$$ U := \frac{1}{4} x^{2/5}, V := \tfrac{1}{2} x^{2/5}.$$
As $x \geq 10^{20}$, the hypotheses of the theorem are obeyed, and we conclude that
\begin{align*}
|S_{\eta,2}(x,\alpha)| &\leq 0.5 \frac{x}{q} \log x \log(x^{4/5}) + 0.89 (\frac{1}{8} x^{4/5} + \frac{5}{2} q) (8 + \log q) \log 2x \\
&\quad +  (0.1 \frac{x}{\sqrt{q}} + 0.39 \frac{x}{\sqrt{x/q}} ) \log(8 x^{1/5})\log(2x) \\
&\quad + 2.3 x^{4/5} \log(2 x^{3/5}).
\end{align*}
As $x \geq 10^{20}$ and $q \leq x/100$, one can compute that
\begin{align*}
\log(8 x^{1/5}) \log(2x) &\leq \frac{1}{5} \log x ( \log x + 11.3 )\\
(8 + \log q) (\log 2x) &\leq 1.1 \log x (\log x + 11.3 )\\
\log(2 x^{3/5}) &\leq 0.011 \log x (\log x + 11.3);
\end{align*}
inserting these bounds and collecting terms, we conclude that
$$
|S_{\eta,2}(x,\alpha)| \leq (0.4 \frac{x}{q} + 0.1 \frac{x}{\sqrt{q}} + 2.45 \frac{x}{x/q} + 0.39 \frac{x}{\sqrt{x/q}} + 0.149 x^{4/5}) \log x (\log x + 11.3).
$$
Since $100 \leq q \leq x/100$, one has
$$ 0.4 \frac{x}{q} \leq 0.04 \frac{x}{\sqrt{q}}$$
and
$$ 2.45 \frac{x}{x/q} \leq 0.245 \frac{x}{\sqrt{x/q}}$$
and the claim \eqref{Sax} follows, after using Lemma \ref{eta-smash} to replace $S_{\eta,2}(x,\alpha)$ with $S_{\eta,q_0}(x,\alpha)$ (at the cost of replacing the $0.149 x^{4/5}$ term with $0.15 x^{4/5}$).

Now we verify \eqref{Sax-2}.  Here we use the choice
$$ U := \frac{x}{q^2}; \quad V := q.$$
The hypotheses of Theorem \ref{strong-minor} are easily verified, and we conclude that
\begin{align*}
|S_{\eta,2}(x,\alpha)| &\leq  
0.5 \frac{x}{q} \log x \log(\frac{2x}{q^2} + 4) + 0.89 (\frac{x}{q}+\frac{5}{2} q) (8+\log q) \log(2x) \\
&\quad + (0.1 \frac{x}{\sqrt{q}} + 0.39 \frac{x}{\sqrt{x/q}} ) \times 3 \log^2 q\\
&\quad + (0.55 \frac{x}{\sqrt{x/q^2}} + 0.78 \frac{x}{\sqrt{q}} ) \times 2 \log q.
\end{align*}
Since $q \leq x^{1/3}$ and $x \geq 10^{20}$, one has
$$ \frac{x}{\sqrt{x/q^2}} \leq \frac{x}{\sqrt{q}}$$
and
$$ \frac{x}{\sqrt{x/q}} \leq 0.001 \frac{x}{\sqrt{q}}$$ 
and
$$ q \leq 0.001 \frac{x}{q}$$
and so 
\begin{align*}
|S_{\eta,2}(x,\alpha)| &\leq  
\frac{x}{q} \log(2x) [ 0.5 \log(\frac{2x}{q^2} + 4) + 0.9 (8+\log q) ] \\
&\quad + (0.301 \log^2 q + 2.66 \log q) \frac{x}{\sqrt{q}}.
\end{align*}
Since
\begin{align*}
0.5 \log(\frac{2x}{q^2} + 4) + 0.9 (8+\log q) 
&\leq 0.5 ( \log( 2x + 4q^2 ) + 14.4 ) \\
&\leq 0.5( \log(2x) + 15 )
\end{align*}
(using the hypotheses $q \leq x^{1/3}$ and $x \geq 10^{20}$) and 
$$ 0.301 \log^2 q + 2.66 \log q \leq 0.301 \log q (\log q + 8.9)$$
we obtain the claim \eqref{Sax-2}, after using Lemma \ref{eta-smash} to replace $S_{\eta,2}(x,\alpha)$ with $S_{\eta,q_0}(x,\alpha)$ (and replacing $0.301$ with $0.31$).

Now we prove \eqref{Sax-3}.  Here we use the choice
$$ U := \frac{x}{(x/q)^2}; \quad V := x/q.$$
Again, the hypotheses of Theorem \ref{strong-minor} are easily verified, and we conclude that
\begin{align*}
|S_{\eta,2}(x,\alpha)| &\leq  
0.5 \frac{x}{q} \log x \log 6 + 0.89 (\frac{7}{2} q) (8+\log q) \log(2x) \\
&\quad + (0.1 \frac{x}{\sqrt{q}} + 0.39 \frac{x}{\sqrt{x/q}} ) \times 3 \log^2 \frac{x}{q}\\
&\quad + (0.55 \frac{x}{\sqrt{x/(x/q)^2}} + 0.78 \frac{x}{\sqrt{x/q}} ) \times 2 \log \frac{x}{q}.
\end{align*}
Since $q \geq x^{2/3}$ and $x \geq 10^{20}$, one has
$$ \frac{x}{\sqrt{x/(x/q)^2}} \leq \frac{x}{\sqrt{x/q}}$$
and
$$ \frac{x}{\sqrt{q}} \leq 0.001 \frac{x}{\sqrt{x/q}}$$
and
$$ \frac{x}{q}\leq 0.001 q$$
and thus
\begin{align*}
|S_{\eta,2}(x,\alpha)| &\leq  3.12 q (8+\log q) \log(2x) \\
&\quad + (1.181 \log^2 \frac{x}{q} + 2.66 \log \frac{x}{q} ) \frac{x}{\sqrt{x/q}}.
\end{align*}
Writing
$$ 1.181 \log^2 \frac{x}{q} + 2.66 \log \frac{x}{q}  \leq 1.181 \log \frac{x}{q} ( \log \frac{x}{q} + 2.3 )$$
we obtain the claim \eqref{Sax-3}, after using Lemma \ref{eta-smash} to replace $S_{\eta,2}(x,\alpha)$ with $S_{\eta,q_0}(x,\alpha)$ (and replacing $1.181$ with $1.19$).

Finally, we establish \eqref{lab}.  Here we use the choice
$$ U := \frac{x}{(1.02 x/q)^2}; V := 1.02 x/q$$
to ensure that $UV < q-1$.  We may now use the term \eqref{term-1-alt} and conclude that
\begin{align*}
|S_{\eta,2}(x,\alpha)| &\leq  
 \frac{96}{\pi^2} \frac{x}{(x/q)^2} \log(4x) \log \frac{4e q}{\pi} \\
&\quad + (0.1 \frac{x}{\sqrt{q}} + 0.39 \frac{x}{\sqrt{x/q}} ) \times 3 \log^2(1.02 x/q)\\
&\quad + (0.55 \frac{x}{\sqrt{x/(1.02 x/q)^2}} + 0.78 \frac{x}{\sqrt{1.02 x/q}} ) \times 2 \log(1.02 x/q).
\end{align*}
We can bound
$$ \log \frac{4e q}{\pi} \leq \log(x/4)$$
and hence
$$ \log(4x) \log(x/4) \leq \log^2 x.$$
Also, as $q \geq x^{1/3}$ and $x \geq 10^{20}$, onehas
$$ \frac{x}{\sqrt{x/(1.02 x/q)^2}} \leq 1.02 \frac{x}{\sqrt{x/q}}$$
and
$$ \frac{x}{\sqrt{q}} \leq 0.001 \frac{x}{\sqrt{x/q}}$$
and thus
\begin{align*}
|S_{\eta,2}(x,\alpha)| &\leq  
 \frac{96}{\pi^2} \frac{x}{(x/q)^2} \log^2 x\\
&\quad + 1.19 \frac{x}{\sqrt{x/q}} \log^2(1.02 x/q)\\
&\quad + 2.67 \frac{x}{\sqrt{x/q}} \log(1.02 x/q).
\end{align*}
The last two terms can be bounded by $(1.19 \log \frac{1.02 x}{q}) \times (\log\frac{1.02 x}{q} +2.3)$.  Since
$$ 1.19 \log \frac{1.02 x}{q} \leq 1.2 \log \frac{x}{q}$$
and
$$ \log \frac{1.02 x}{q} + 2.3 \leq \log \frac{x}{q} + 2.35$$
and
$$ \frac{96}{\pi^2} \leq 9.73$$
the claim \eqref{lab} follows, after using Lemma \ref{eta-smash} to replace $S_{\eta,2}(x,\alpha)$ with $S_{\eta,q_0}(x,\alpha)$ (and replacing $2.35$ with $2.4$).

\section{Major arc estimate}\label{major-sec}

In this section we use Theorem \ref{rh-check} to control exponential sums in the ``major arc'' regime $\alpha = O(T_0/x)$.  To link the von Mangoldt function to the zeroes of the zeta function we use the following standard identity:

\begin{proposition}[Von Mangoldt explicit formula]  Let $\eta: \R \to \C$ be a smooth, compactly supported function supported in $[2,+\infty)$.  Then
\begin{equation}\label{laman} 
\sum_n \Lambda(n) \eta(n) = \int_\R (1 - \frac{1}{x^3-x}) \eta(y)\ dy - \sum_\rho \int_\R \eta(y) y^{\rho-1}\ dy 
\end{equation}
where $\rho$ ranges over the non-trivial zeroes of the Riemann zeta function.
\end{proposition}

\begin{proof} See, for instance, \cite[\S 5.5]{iwaniec}.
\end{proof}

\begin{proposition}[Major arc sums]\label{rh}  Let $T_0$ be as in Theorem \ref{rh-check}.  Let $\eta: \R \to \R^+$ be a smooth non-negative function supported on $[c,c']$, and let $x, \alpha \in \R$ be such that $cx \geq 10^3$ and 
\begin{equation}\label{aleph}
|\alpha| \leq \frac{T_0}{4\pi c'x}.  
\end{equation}
Then
\begin{equation}\label{exam}
|S_{\eta,1}(x,\alpha) - x \int_\R \eta(y) e(\alpha x y)\ dy| \leq 
A \frac{\log T_0}{3T_0} x + 2.01 c^{-1/2} x^{1/2} N(T_0)\|\eta\|_{L^1(\R)},
\end{equation}
where $A$ is the quantity
\begin{equation}\label{A-def}
 A := 60 \|\eta\|_{L^1(\R)} + 32 c' \| \eta'\|_{L^1(\R)} + 4 (c')^2 \| \eta''\|_{L^1(\R)}
\end{equation}
and $N(T_0)$ is the number of zeroes of $\zeta$ in the strip $\{ 0 \leq \Re(s) \leq 1;0\leq \Im(s) \leq T_0\}$.
\end{proposition}

As should be clear from the proof, the constant $A$ can be improved somewhat by a more careful argument, for instance by exploiting the functional equation for $\zeta$, which places at most half of the zeroes to the right of the critical line $\Re \rho = 1/2$.  However, this does not end up being of much significance, as the right-hand side of \eqref{exam} is already small enough for our purposes (thanks to the large denominator of $T_0$).  It may appear odd that the condition \eqref{aleph} allows for $|\alpha|$ to exceed $1$, but the estimate \eqref{exam} becomes weaker than the trivial bound in Lemma \ref{rs-lemma} in this case.

For our particular choice of $T_0$, namely $T_0 := 3.29 \times 10^9$, the quantity $N(T_0)$ can be bounded by $10^{10}$ (indeed, the value of $T_0$ we have chosen arose from the verification that the first $10^{10}$ zeroes of $\zeta$ were on the critical line).

\begin{proof}  We have
$$ S_{\eta,1}(x,\alpha) = \sum_n \Lambda(n) \eta(n/x) e(\alpha n).$$
Applying \eqref{laman}, we conclude that the left-hand side of \eqref{exam} is bounded by
$$
\int_\R \frac{1}{y^3-y} \eta(y/x)\ dy + \sum_\rho \left|\int_\R \eta(y/x) e(\alpha y) y^{\rho-1}\ dy\right| 
$$
Crudely bounding $\frac{1}{y^3-y}$ by $\frac{2}{(cx)^3}$ on the support of $\eta(y/x)$, the first term is at most $2c^{-3} x^{-2} \|\eta\|_{L^1(\R)}$.   Now we turn to the second sum.  We write $\rho = \sigma+it$ with $0 <\sigma<1$ and $t \in \R$.  

We first consider the contribution of those zeroes with $|t| \leq T_0$.  By hypothesis, $\sigma=1/2$ for these zeroes, and we may bound this portion of the sum by
$$ \sum_{\rho: |t| \leq T_0} \int_\R \eta(y/x) y^{-1/2}\ dy$$
which we may bound in turn by
$$ 2 c^{-1/2} x^{1/2} N(T_0) \| \eta\|_{L^1(\R)}.$$
As $T_0 \geq 10^3$ and $cx \geq 10^3$, we may clearly absorb the much smaller error term $2c^{-3} x^{-2} \|\eta\|_{L^1(\R)}$ into this term by increasing the $2$ factor to $2.01$.

Finally, we consider the terms with $|t| > T_0$.  We may rewrite a single integral
\begin{equation}\label{roar}
 |\int_\R \eta(y/x) e(\alpha y) y^{\rho-1}\ dy|
\end{equation}
as
$$ |\int_\R f(y) e^{i(2\pi \alpha y + t \log y)}\ dy|$$
where $f(y) := \eta(y/x) y^{\sigma-1}$.  Since
$$ e^{i(2\pi \alpha y + t \log y)} = \frac{1}{i (2\pi \alpha + t/y)} \frac{d}{dy} e^{i(2\pi \alpha y + t \log y)}$$
with the denominator non-vanishing on the support of $\eta(y/x)$ thanks to \eqref{aleph}, we may integrate by parts and write the preceding integral as
$$ |\int_\R \frac{d}{dy}\left( \frac{1}{2\pi \alpha + t/y} f(y) \right) e^{i(2\pi \alpha y + t \log y)}\ dy|.$$
A second integration by parts then rewrites the above expression as
$$ |\int_\R \frac{d}{dy} \left( \frac{1}{2\pi \alpha + t/y} \frac{d}{dy}\left( \frac{1}{2\pi \alpha + t/y} f(y) \right) \right) e^{i(2\pi \alpha y + t \log y)}\ dy|$$
which we then bound by
$$ \int_\R |\frac{d}{dy} \left( \frac{1}{2\pi \alpha + t/y} \frac{d}{dy}\left( \frac{1}{2\pi \alpha + t/y} f(y) \right) \right)|\ dy.$$
The expression inside the absolute value can be expanded as
\begin{align*}
& \frac{t^2/y^4 - 4\pi \alpha t/y^3}{(2\pi \alpha+t/y)^4} f(y) \\
&\quad + \frac{3t/y^2}{(2\pi \alpha+t/y)^3} f'(y)\\
&\quad + \frac{1}{(2\pi \alpha+t/y)^2} f''(y).
\end{align*}
On the support of $f$ one has $y \leq c' x$.  From \eqref{aleph} we may thus lower bound
$$ |2\pi \alpha+t/y| \geq \frac{|t|}{2y}$$
and upper bound
$$ |t^2/y^4 - 4\pi \alpha t/y^3| \leq \frac{2t^2}{y^4},$$
and so the preceding integral may be bounded by
\begin{align*}
&\frac{32}{t^2}\int_\R |f(y)| dy \\
&\quad + \frac{24}{t^2} \int_\R y |f'(y)|\ dy \\
&\quad + \frac{4}{t^2} \int_\R y^2 |f''(y)|\ dy.
\end{align*}
Since $0 \leq \sigma \leq 1$, we have the crude bounds
\begin{align*}
|f(y)| &\leq \eta(y/x) \\
|f'(y)| &\leq \frac{1}{x} |\eta'(y/x)| + \frac{1}{y} \eta(y/x) \\
|f''(y)| &\leq \frac{1}{x^2} |\eta''(y/x)| + \frac{2}{xy} |\eta'(y/x)| + \frac{1}{y^2} \eta(y/x)
\end{align*}
and so (on bounding $y$ from above by $c'x$) we may bound \eqref{roar} by
$$ 
A \frac{x}{t^2}$$
where $A$ is the quantity defined in \eqref{A-def}.
From \cite[Lemma 2]{saouter}, we may bound
$$ \sum_{\rho: |t| \geq T_0} \frac{1}{t^2}
\leq \frac{1}{\pi T_0} (\log \frac{T_0}{2\pi} + 1) + \frac{1.34}{T_0^2} ( 2 \log \frac{T_0}{2\pi} + 1),$$
which under the hypothesis $T_0 \geq 10^3$ can be bounded above by
$$ \frac{\log T_0}{3T_0}.$$
The claim then follows.
\end{proof}

\section{Sums of five primes}\label{final-sec}

In this section we establish Theorem \ref{main}.  We recall the following result of Ramar\'e and Saouter:

\begin{theorem}  If $x \geq  1.1 \times 10^{10}$ is a real number, then there is at least one prime $p \leq x$ with $x-p \leq \frac{x}{2.8 \times 10^7}$.
\end{theorem}

\begin{proof}  See \cite[Theorem 3]{saouter}.  One of the main tools used in the proof is Theorem \ref{rh-check}.  By using the more recent numerical verifications of the Riemann hypothesis, one can improve this result; see \cite{helf}.  However, we will not need to do so here.
\end{proof}

By combining the above result with Theorem \ref{gh-check}, we see that

\begin{itemize}
\item Every odd number between $3$ and $(2.8 \times 10^7) N_0$ is the sum of at most three primes.
\item Every even number between $2$ and $(2.8 \times 10^7)^2 N_0$ is the sum of at most four primes.
\item Every odd number between $3$ and $(2.8 \times 10^7)^3 N_0$ is the sum of at most five primes.
\end{itemize}

In particular, Theorem \ref{main} holds up to $(2.8 \times 10^7)^3 N_0 \geq 8.7 \times 10^{36}$.  Thus, it suffices to verify Theorem \ref{main} for odd numbers larger than $8.7 \times 10^{36}$.

On the other hand, in \cite{liu} it is shown that every odd number larger than $\exp(3100)$ is the sum of three primes.  It will then suffice to show that

\begin{theorem}\label{mod}  Let $8.7 \times 10^{36} \leq x \leq \exp(3100)$ be an integer.  Then there is an integer in the interval $[x-N_0,x-2]$ which is the sum of three odd primes.
\end{theorem}

We establish this theorem by the circle method.  Fix $x$ as above.  We will need a symmetric, Lipschitz, $L^2$-normalised cutoff function $\eta_1$ which is close to $\1_{[0,1]}$.  For sake of concreteness we will take
$$ \eta_1(t) := (1 - 10\operatorname{dist}(t,[0.2,0.8]))_+,$$
which is supported on $[0.1,0.9]$ and obeys the symmetry
\begin{equation}\label{eta-sym}
\eta_1(1-t) = \eta_1(t) \hbox{ for all } t \in \R.
\end{equation}
For future reference we record some norms on $\eta_1$ (after performing an infinitesimal mollification):
\begin{align}
\| \eta_1\|_{L^2(\R)} &= \sqrt{\frac{2}{3}} \label{norm}\\
\| \eta_1\|_{L^\infty(\R)} &= 1 \label{sqr0}\\
\| \eta_1 \|_{L^1(\R)} &= \frac{7}{10} \label{sqrf} \\
\| \eta'_1\|_{L^\infty(\R)} &= 10 \label{sqr1}\\
\| \eta'_1 \|_{L^1(\R)} &= 2 \label{sqrf-2} \\
\| \eta_1 \eta'_1 \|_{L^1(\R)} &= 1 \label{sqrf-11} \\
\| \eta''_1 \|_{L^1(\R)} &= 40 \label{sqrf-3} \\
\| \eta'_1 \eta'_1 + \eta_1 \eta''_1\|_{L^1(\R)} &= 40. \label{sqrf-16}
\end{align}

For a parameter $K \geq 1$ to be chosen later, we will consider the quantity\footnote{The ranges of the $h_1,h_2,h_3$ parameters are not optimal; one could do a bit better, for instance, by requiring $1 \leq h_1,h_2 \leq N_0/20$ and $1 \leq h_3 \leq 0.9 N_0$ instead, as this barely impacts Proposition \ref{loot}, but reduces the upper bound in \eqref{mech} by almost a factor of three.  However, we will not need this improvement here.}
\begin{equation}\label{quant}
\begin{split}
\sum_{n_1,n_2,n_3} \sum_{1 \leq h_1,h_2,h_3 \leq N_0/3} &\Lambda(n_1) \1_{(n_1,\sqrt{x}\sharp)=1} \eta_1( n_1 / x ) \Lambda(n_2) \1_{(n_2,\sqrt{x}\sharp)=1} \eta_1( n_2 / x ) \\
&\quad \times \Lambda(n_3) \1_{(n_3,\sqrt{x/K}\sharp)=1} \eta_0( K n_3 / x ) \1_{x=n_1+n_2+n_3+h_1+h_2+h_3}
\end{split}
\end{equation}
Observe that the summand is only non-zero when $n_1,n_2,n_3$ are odd primes (of magnitudes less than $x$, $x$, and $x/K$ respectively) with $n_1+n_2+n_3 \in [x-N_0,x-2]$.  Thus, to prove Theorem \ref{mod}, it suffices to show that the quantity \eqref{quant} is strictly positive.  Applying the prime number heuristic $\Lambda \approx 1$ and \eqref{norm}, we see that we expect the quantity \eqref{quant} to be of size roughly $\frac{2}{3} x^3 K^{-1} (N_0/3)^3$.

Note that we have made the third prime $n_3$ somewhat smaller than the other two.  This trick, originally due to Bourgain \cite{bourgain}, helps remove some losses associated to the convolution of the cutoff functions $\eta_1, \eta_0$.  

By Fourier analysis, we may rewrite \eqref{quant} as
\begin{equation}\label{nonz}
\int_{\R/\Z} S_{\eta_1,\sqrt{x}\sharp}(x,\alpha)^2 S_{\eta_0,\sqrt{x/K}\sharp}(x/K,\alpha) D_{N_0/3}(\alpha)^3 e(-x\alpha)\ d\alpha
\end{equation}
where $D_{N_0/2}(\alpha)$ is the Dirichlet-type kernel
$$ D_{N_0/2}(\alpha) := \sum_{1 \leq n \leq N_0/3} e(n \alpha).$$
It thus suffices to show that the expression is non-zero for some $K$.  It turns out that there is some gain to be obtained (of about an order of magnitude) in the weakly minor arc regime when $\alpha$ is slightly larger than $T_0/x$, by integrating $K$ over an interval inside $[1,T_0]$ that is well separated from both endpoints; for instance, one could integrate $K$ from $10^3$ to $10^6$.  Such a gain could be useful for further work on Goldbach type problems; however, for the problem at hand, it will be sufficient to select a single value of $K$, namely
$$ K := 10^3.$$

As mentioned previously, the expected size of the quantity \eqref{nonz} is roughly $\frac{2}{3} \frac{x^2}{K} (N_0/3)^3$.  This heuristic can be supported by the following major arc estimate:

\begin{proposition}[Strongly major arc estimate]\label{smae}  We have
\begin{equation}\label{mamma}
\begin{split}
 \int_{\|\alpha\|_{\R/\Z} \leq \frac{T_0}{3.6 \pi x}} &S_{\eta_1,\sqrt{x}\sharp}(x,\alpha)^2 S_{\eta_0,\sqrt{x/K}\sharp}(x/K,\alpha) D_{N_0/3}(\alpha)^3 e(-x\alpha) \ d\alpha\\
&\quad  = 
\frac{x^2}{K} (N_0/3)^3 (\frac{2}{3} + {\mathcal O}^*(0.1)).
\end{split}
\end{equation} 
\end{proposition}

\begin{proof}  
From the mean value theorem one has
$$ D_{N_0/3}(\alpha) = (N_0/3) (1 + {\mathcal O}^*( 2 \pi(N_0/3) \|\alpha\|_{\R/\Z})) = (N_0/3) (1 + {\mathcal O}^*(\frac{N_0 T_0}{5.4 x} ))
$$
when $\|\alpha\|_{\R/\Z} \leq \frac{T_0}{3.6 \pi x}$.
Since $N_0 = 4 \times 10^{14}$, $T_0 = 3.29 \times 10^9$, and $x \geq 8.7 \times 10^{36}$, we conclude that
\begin{equation}\label{dna}
D_{N_0/3}(\alpha) = (N_0/2) (1 + {\mathcal O}^*(10^{-10}))
\end{equation}
(with plenty of room to spare).  Also, by Proposition \ref{rh}, for $\alpha$ in the interval $[-\frac{T_0}{3.6 \pi x}, \frac{T_0}{3.6 \pi x}]$ (and in fact for all $\alpha$ in the wider interval $[-\frac{KT_0}{4\pi x}, \frac{KT_0}{4\pi x}]$), one has
$$ S_{\eta_0,1}(x/K,\alpha) = \frac{x}{K} \hat \eta_0(\alpha x/K) + {\mathcal O}^*( 
A_2 \frac{\log T_0}{3T_0} \frac{x}{K} + 2.01 x^{1/2} K^{-1/2} N(T_0)\|\eta_0\|_{L^1(\R)} )$$
where $\hat \eta_0(\alpha) := \int_\R \eta_0(y) e(\alpha x y)\ dy$ and
$$ A_2 := 60 \|\eta_0\|_{L^1(\R)} + 32 \| \eta'_0\|_{L^1(\R)} + 4 \| \eta''_0\|_{L^1(\R)}.$$
From \eqref{s1}, \eqref{s1-3}, \eqref{s1-5} one has 
$$ A_2 = 252 + 256 \log 2 \leq 330$$
and with $x \geq 8.7 \times 10^{36}$, $K = 10^3$, $T_0 = 3.29 \times 10^9$, and $N(T_0) \leq 10^{10}$, we conclude that
$$  S_{\eta_0,1}(x/K,\alpha) = \frac{x}{K} ( \hat \eta_0(\alpha x/K) + {\mathcal O}^*( 10^{-6} ) ).$$

By Lemma \ref{eta-smash} and \eqref{dna} (and bounding $\hat \eta_0(\alpha x)$ as ${\mathcal O}^*(1)$ whenever necessary) we then have
$$  S_{\eta_0,\sqrt{x/K}\sharp}(x/K,\alpha) D_{N_0/3}(\alpha)^3 = \frac{x (N_0/3)^3}{K} ( \hat \eta_0(\alpha x) + {\mathcal O}^*( 1.1 \times 10^{-6} ) ).$$
Let us consider the contribution of the error term $\frac{x (N_0/3)^3}{K} {\mathcal O}^*( 1.1 \times 10^{-6} )$ to \eqref{mamma}.  By Corollary \ref{uplow}, we may bound this contribution in magnitude by
$$ 1.1 \times 10^{-6} \frac{x(N_0/3)^3}{K} \frac{2}{1-\frac{\log(2T_0/3\pi)}{\log x}} S_{\eta_1^2,\sqrt{x/K}\sharp}(x,0)$$
which by Lemma \ref{rs-lemma} and the bounds $T_0 = 3.29 \times 10^9$, $x \geq 8.7 \times 10^{36}$ can be safely bounded by
$$ 10^{-3} \frac{x^2(N_0/3)^3}{K}.$$
Thus it will suffice to show that
\begin{equation}\label{nom}
\int_{|\alpha| \leq \frac{T_0}{3\pi x}} S_{\eta_1,\sqrt{x}\sharp}(x,\alpha)^2 \hat \eta_0(\alpha x/K) e(-x\alpha)\ d\alpha
 =  x (\frac{2}{3} + {\mathcal O}^*(0.09)).
\end{equation}

We apply Proposition \ref{rh} again to obtain
$$ S_{\eta_1,1}(x,\alpha) = x \hat \eta_1(\alpha x) + {\mathcal O}^*( 
A_1 \frac{\log T_0}{3T_0} x + 2.01 c_1^{-1/2} x^{1/2} N(T_0)\|\eta_1\|_{L^1(\R)} )$$
where $\hat \eta_1(\alpha) := \int_\R \eta_1(y) e(\alpha x y)\ dy$ and
$$
 A_1 := 60 \|\eta_1\|_{L^1(\R)} + 32 c'_1 \| \eta'_1\|_{L^1(\R)} + 4 (c'_1)^2 \| \eta''_1\|_{L^1(\R)}
$$
with $c_1 := 0.1$, $c'_1 = 0.9$. From \eqref{sqrf}, \eqref{sqrf-2}, \eqref{sqrf-3} one has
$$ A_1 = 229.2$$
and with $x \geq 8.7 \times 10^{36}$, $T_0 = 3.29 \times 10^9$, and $N(T_0) \leq 10^{10}$, we conclude that
$$  S_{\eta_1,1}(x,\alpha) = x ( \hat \eta_1(\alpha x) + {\mathcal O}^*( 10^{-6} ) ).$$
Bounding $\hat \eta_0$ as ${\mathcal O}^*(1)$, we may thus express the left-hand side of \eqref{nom} as the sum of the main term
$$ x^2 \int_{|\alpha| \leq \frac{T_0}{3.6 \pi x}} \hat \eta_1(\alpha x)^2 \hat \eta_0(\alpha x/K) e(-x\alpha)\ d\alpha$$
and the two error terms
\begin{equation}\label{thing-1}
{\mathcal O}^*( 10^{-6} x \int_{|\alpha| \leq \frac{T_0}{3.6\pi x}} |S_{\eta_1,1}(x,\alpha)|\ d\alpha )
\end{equation}
and
\begin{equation}\label{thing-2}
{\mathcal O}^*( 10^{-6} x^2 \int_{|\alpha| \leq \frac{T_0}{3.6\pi x}} |\hat \eta_1(\alpha x)|\ d\alpha ).
\end{equation}
We first estimate \eqref{thing-1}.  From Corollary \ref{uplow} and Lemma \ref{rs-lemma} we have
$$
\int_{|\alpha| \leq \frac{T_0}{3.6 \pi x}} |S_{\eta_1,1}(x,\alpha)|^2\ d\alpha
\leq \frac{2}{1-\frac{\log(2T_0/3.6\pi)}{\log x}} \times 4 \log 2 \times 1.04 x$$
and so by Cauchy-Schwarz we may bound \eqref{thing-1} by
$$
10^{-6} x \sqrt{\frac{T_0}{3.6\pi}} (\frac{2}{1-\frac{\log(2T_0/3.6\pi)}{\log x}} \times 4 \log 2 \times 1.04)^{1/2};$$
since $x \geq 8.7 \times 10^{36}$, $T_0 = 3.29 \times 10^9$, we may bound \eqref{thing-1} by $0.02 x$.  A similar (in fact, slightly better) bound obtains for \eqref{thing-2} (using Plancherel's theorem in place of Corollary \ref{uplow}).  We conclude that it will suffice to show that
$$
\int_{|\alpha| \leq \frac{T_0}{3.6 \pi x}} \hat \eta_1(\alpha x)^2 \hat \eta_0(\alpha x/K) e(-x\alpha)\ d\alpha = \frac{1}{x} (1 + {\mathcal O}^*( 0.05 )).$$
By Lemma \ref{inter} we have
$$ |\hat \eta_1(\alpha x)| \leq \frac{\|\eta'_1\|_{L^1(\R)}}{2\pi|\alpha|} = \frac{1}{\pi|\alpha|};$$
from this and the choice $T_0 = 3.29 \times 10^9$ one easily verifies that
$$
\int_{|\alpha| > \frac{T_0}{3.6 \pi x}} \hat \eta_1(\alpha x)^2 \hat \eta_0(\alpha x/K) e(-x\alpha)\ d\alpha = {\mathcal O}^*(0.01 \frac{1}{x})$$
and so it remains to show that
$$
\int_\R \hat \eta_1(\alpha x)^2 \hat \eta_0(\alpha x/K) e(-x\alpha)\ d\alpha = \frac{1}{x} (1 + {\mathcal O}^*( 0.04 )).$$
The left-hand side can be expressed as
$$ \frac{1}{x} \int_\R \int_\R \eta_1(s) \eta_1(1-s-t/K) \eta_0(t)\ ds dt;$$
as $\eta_0$ has an $L^1$ norm of $1$, it thus suffices to show that
$$ \int_\R \eta_1(s) \eta_1(1-s-t/K)\ ds = 1 + {\mathcal O}^*(0.04)$$
for all $t = {\mathcal O}^*(1)$.  By \eqref{eta-sym} and \eqref{sqr1} one has
$$ \eta_1(1-s-t/K) = \eta_1(s) + {\mathcal O}^*( 10 / K )$$
and so (by \eqref{sqrf} and the $L^1$-normalisation of $\eta_0$)
$$ \int_\R \eta_1(s) \eta_1(1-s-t/K)\ ds = 1 + {\mathcal O}^*(10/K).$$
Since $K = 10^3$, the claim follows. 
\end{proof}

In view of the above proposition, it suffices to show that
\begin{equation}\label{semon}
 \int_{\|\alpha\|_{\R/\Z} \geq \frac{T_0}{3.6 \pi x}} |S_{\eta_1,\sqrt{x}\sharp}(x,\alpha)|^2 |S_{\eta_0,\sqrt{x/K}\sharp}(x/K,\alpha)| |D_{N_0/3}(\alpha)|^3\ d\alpha
 \leq 0.56  \frac{x^2}{K} (N_0/3)^2.
\end{equation}

We have the following $L^2$ estimate:

\begin{proposition}[$L^2$ estimate]\label{loot}  We have
$$
 \int_{\|\alpha\|_{\R/\Z} \geq \frac{T_0}{3.6 \pi x}} |S_{\eta_1,\sqrt{x}\sharp}(x,\alpha)|^2 |D_{N_0/3}(\alpha)|^2\ d\alpha
\leq 7.09 (N_0/3)^2 x.$$
\end{proposition}

\begin{proof} By Proposition \ref{meso} we have
$$ \int_{\R/\Z} |S_{\eta_1,\sqrt{x}\sharp}(x,\alpha)|^2 |D_{N_0/3}(\alpha)|^2 \ d\alpha \leq 8 H^2 x \left( 1 + \frac{0.13 \log x}{H} + \frac{(e^\gamma \log \log(2H) + \frac{2.507}{\log\log(2H)}) \log(9H)}{2H} \right).
$$
where $H := \lfloor N_0/3\rfloor$.  Since $\log x \leq 3100$ and $N_0 = 4 \times 10^{14}$, a brief computation then shows that
$$ \int_{\R/\Z} |S_{\eta_1,\sqrt{x}\sharp}(x,\alpha)|^2 |D_{N_0/3}(\alpha)|^2 \ d\alpha \leq 8.001 (N_0/3)^2 x.$$
Meanwhile, from Corollary \ref{downlow} (using the values $x \geq 8.7 \times 10^{36}$, $T_0 = 3.29 \times 10^9$, $c = 0.1$, and the estimates
\eqref{sqr0}, \eqref{sqrf-11}, \eqref{sqrf-16} to verify the hypotheses of that corollary), we have
$$ \int_{\|\alpha\|_{\R/\Z} \geq \frac{T_0}{3.6\pi x}} |S_{\eta_1,\sqrt{x}\sharp}(x,\alpha)|^2 \ d\alpha \geq 0.92 x.$$
and thus by \eqref{dna}
$$ \int_{\|\alpha\|_{\R/\Z} \geq \frac{T_0}{3.6\pi x}} |S_{\eta_1,\sqrt{x}\sharp}(x,\alpha)|^2 |D_{N_0/3}(\alpha)|^2 \ d\alpha \geq 0.919 (N_0/3)^2 x.$$
Subtracting, one obtains the bound.
\end{proof}

In view of this proposition and H\"older's inequality, it suffices to show the $L^\infty$ estimate that
\begin{equation}\label{nao3}
|S_{\eta_0,\sqrt{x/K}\sharp}(x/K,\alpha)| |D_{N_0/3}(\alpha)| \leq 0.078 x (N_0/3) \frac{x}{K}
\end{equation}
whenever $\|\alpha\|_{\R/\Z} \geq \frac{T_0}{3.6\pi x}$; comparing this with Lemma \ref{rs-lemma}, we see that we have to beat the ``trivial'' bounds in that lemma by a factor of roughly $13$.  By the conjugation symmetry \eqref{conj} we may assume that
$$ \frac{T_0}{3\pi x} \leq \alpha \leq \tfrac{1}{2}.$$

We first consider the case of the weakly major arc regime
$$ \frac{T_0}{3.6\pi x} < \alpha \leq \frac{K T_0}{4\pi x}.$$
By the computations in the proof of Proposition \ref{smae}, we have
$$  S_{\eta_0,\sqrt{x/K}\sharp}(x/K,\alpha) = \frac{x}{K} ( \hat \eta_0(\alpha x/K) + {\mathcal O}^*( 1.1 \times 10^{-6} ) )$$
in this regime, so by the trivial bound $|D_{N_0/3}(\alpha)| \leq (N_0/3)$ it suffices to show that
$$ |\hat \eta_0(\alpha x/K)| \leq 0.077.$$
By Lemma \ref{inter}, we have
$$ |\hat \eta_0(\alpha x/K)| \leq \frac{\|\eta'_0\|_{L^1(\R)}}{2\pi|\alpha| x/K}.$$
Since $\alpha \geq \frac{T_0}{3.6\pi x}$, we conclude from \eqref{s1-2} that
$$ |\hat \eta_0(\alpha x/K)| \leq \frac{7.2 K \log 2}{T_0},$$
and the claim follows (with plenty of room to spare) from the choices $K = 10^3$, $T_0 = 3.29 \times 10^9$.  

Next, we observe from Lemma \ref{rs-lemma} and \eqref{s1-3} (and the lower bound $x/K \geq 8.7 \times 10^{33}$) that
$$ |S_{\eta_0,\sqrt{x/K}\sharp}(x/K,\alpha)| \leq S_{\eta_0,1}(x/K,0) \leq 1.01 \frac{x}{K};$$
combining this with the crude bound
$$ |D_{N_0/3}(\alpha)| \leq \frac{2}{|1-e(\alpha)|} = \frac{1}{|\sin(\pi \alpha)|} \leq \frac{1}{2\alpha}$$
(by \eqref{sinx}), we have
$$|S_{\eta_0,\sqrt{x/K}\sharp}(x/K,\alpha)| |D_{N_0/3}(\alpha)|  \leq \frac{1.01}{2\alpha} \frac{x}{K}$$
which gives \eqref{nao3} whenever 
$$ \alpha \geq \frac{20}{N_0}$$
and so we may assume that
\begin{equation}\label{mech}
 \frac{K T_0}{3\pi x} \leq \alpha < \frac{20}{N_0}.
\end{equation}
In this regime we use the crude bound $|D_{N_0/3}(\alpha)| \leq N_0/3$, and reduce to showing that
\begin{equation}\label{noose}
 |S_{\eta_0,\sqrt{x/K}\sharp}(x/K,\alpha)| \leq 0.078 x.
\end{equation}
We may write $4\alpha = \frac{1}{q} + {\mathcal O}^*(1/q^2)$ for some 
$$ \frac{1}{4\alpha}-1 \leq q \leq \frac{1}{4\alpha};$$ 
in particular
$$ \frac{N_0}{80}-1 \leq q \leq \frac{\pi x}{K T_0}.$$

We first consider the weakly minor arc case
$$ (x/K)^{2/3} \leq q \leq \frac{\pi x}{K T_0}.$$
In this case we may use \eqref{lab} to bound the left-hand side of \eqref{noose} by
$$ 
(9.73 \frac{1}{(x/Kq)^2} \log^2(x/K) + 1.2 \frac{1}{\sqrt{x/Kq}} \log \frac{x}{Kq} (\log \frac{x}{Kq} + 2.4 )) \frac{x}{K}.$$
Since $x/q \geq T_0/\pi$, we can bound this by
$$ 
(9.73 (\frac{\pi}{T_0})^2 \log^2 x + 1.19 \frac{\sqrt{\pi}}{\sqrt{T_0}} \log(T_0/\pi) (\log(T_0/\pi) + 2.3 )) \frac{x}{K}.$$
As $\log x \leq 3100$ and $T_0 = 3.29 \times 10^9$, the expression inside the parentheses is certainly less than $0.078$ (in fact it is less than $0.004$).  Note here that the computations would not have worked if we had used the weaker bounds \eqref{Sax}, \eqref{Sax-3} in place of \eqref{lab}.

Next, we consider the intermediate minor arc case
$$ (x/K)^{1/3} \leq q \leq (x/K)^{2/3}.$$
Here, we use \eqref{Sax} to bound the left-hand side of \eqref{noose} by
$$ (0.14 \frac{y}{\sqrt{q}} + 0.64 \frac{1}{\sqrt{y/q}} + 0.15 y^{-1/5}) \log y (\log y + 11.3) \times \frac{x}{K}.$$
where $y := x/K$.  Using the bounds on $q$, this can be bounded in turn by
$$ (0.8 y^{-1/6} + 0.027 y^{-1/5}) \log y (\log y + 11.3) \times \frac{x}{K}.$$
Since $y \geq 8.7 \times 10^{33}$, the expression $(0.8 y^{-1/6} + 0.15 y^{-1/5}) \log y (\log y + 11.3)$ is bounded by $0.078$ (in fact it is less than $0.013$), so this case is also acceptable.

Finally, we consider the strongly minor arc case
$$ \frac{N_0}{80}-1 \leq q \leq (x/K)^{1/3}.$$
Note that this case can be vacuous if $x$ is too small.  In this case we use \eqref{Sax-2} to bound the left-hand side of \eqref{noose} by
$$
[ 0.5 \frac{1}{q} \log(2x/K) (\log(2x/K) + 15) + 0.31 \frac{1}{\sqrt{q}} \log q (\log q + 8.9) ] \times \frac{x}{K}.$$
Since $\log(2x/K) \leq \log x \leq 3100$ and $q \geq \frac{N_0}{80}-1 = 5 \times 10^{12} - 1$, the expression in brackets is bounded by $0.078$ (in fact it is less than $0.0002$).  This concludes the proof of \eqref{nao3} in all cases, and hence of Theorem \ref{main}.

\end{document}